\documentclass[11pt, a4paper, reqno]{amsart}
\DeclareMathAlphabet{\mathscrbf}{OMS}{mdugm}{b}{n}

\usepackage[utf8]{inputenc}
\usepackage[slovene, english]{babel}
\usepackage{polski}
\usepackage{amsmath,amsfonts,amssymb,mathrsfs,amsthm,amscd, ulem}
\usepackage[pdftex]{graphicx}
\usepackage{tikz-cd}
\usepackage{soul} 
\usepackage{MnSymbol}
\usepackage[all,cmtip]{xy}
\usepackage{enumitem}

\usepackage{slashed}

\usepackage{pgfplots}

\usepackage[mathcal]{euscript}  % the nice mathcal font

\usepackage{url}

\usetikzlibrary{arrows}

 \usepackage{relsize}

\tikzset{
  no line/.style={draw=none,
    commutative diagrams/every label/.append style={/tikz/auto=false}},
  from/.style args={#1 to #2}{to path={(#1)--(#2)\tikztonodes}}}

\title[Moduli of oriented formal groups and the chromatic filtration]{
Moduli stack of oriented formal groups\\ and the chromatic filtration
}
\author{Rok Gregoric}
\thanks{University of Texas at Austin}
\date{\today}
\address{Department of Mathematics, University of Texas at Austin, Austin, TX 78712, USA}
\email{gregoric@math.utexas.edu}

\newtheorem{theorem}{Theorem}[subsection]
\newtheorem{theoremm}{Theorem}

\newtheorem{corollary}[theorem]{Corollary}
\newtheorem{lemma}[theorem]{Lemma}
\newtheorem{prop}[theorem]{Proposition}

\theoremstyle{definition}
\newtheorem{definition}[theorem]{Definition}

\newtheorem{cons}[theorem]{Construction}
\newtheorem{conns}[theoremm]{Construction}

\newtheorem{ex}[theorem]{Examples}
\newtheorem{exun}[theorem]{Example}
\newtheorem{remark}[theorem]{Remark}
\newtheorem{remmark}[theoremm]{Remark}

\usepackage{tikz, calc}
\usetikzlibrary{matrix,arrows}

\newcommand*{\Cat}{\mathcal C\mathrm{at}_\infty}

\newcommand*{\CAlg}{{\operatorname{CAlg}}}
\newcommand*{\CAlgcn}{{\operatorname{CAlg^{cn}}}}

\newcommand*{\F}{\overline{\mathbf{F}}_p}

\newcommand*{\mC}{\mathcal C}

\newcommand*{\mS}{\mathcal S}

\newcommand*{\sO}{\mathcal O}
\newcommand*{\sF}{\mathscr F}

\newcommand*{\sG}{\mathscr G}

\newcommand*{\E}{\mathbb E_\infty}

\newcommand*{\heart}{\heartsuit}

\newcommand*{\sheafhom}{\mathscr{H}\kern -.5pt om}
\DeclareMathOperator{\Novak}{\mathscr{N}\text{\kern -3pt {\calligra\large ovak}}\,\,}

\usepackage{amsmath,calligra,mathrsfs}
\DeclareMathOperator{\fHom}{\mathscr{H}\text{\kern -3pt {\calligra\large om}}\,}
\DeclareMathOperator{\id}{\operatorname{id}}

\DeclareMathOperator{\Sp}{\operatorname{Sp}}

\DeclareMathOperator{\Fun}{\operatorname{Fun}}

\DeclareMathOperator{\Spec}{\operatorname{Spec}}
\DeclareMathOperator{\Spf}{\operatorname{Spf}}

\DeclareMathOperator{\Shv}{\mathcal S\mathrm{hv}}

\DeclareMathOperator{\Map}{\operatorname{Map}}
\DeclareMathOperator{\QCoh}{\operatorname{QCoh}}
\DeclareMathOperator{\IndCoh}{\operatorname{IndCoh}}

\DeclareMathOperator{\Tot}{\operatorname{Tot}}

\DeclareMathOperator{\fib}{\operatorname{fib}}

\DeclareMathOperator{\MP}{\mathrm{MP}}

\DeclareMathOperator{\M}{\mathcal M^\mathrm{or}_\mathrm{FG}}
\DeclareMathOperator{\Mo}{\mathcal M^\heart_\mathrm{FG}}
\DeclareMathOperator{\Mn}{\mathcal M^{\mathrm{or}, \le n}_\mathrm{FG}}
\DeclareMathOperator{\Mon}{\mathcal M^{\heart, \le n}_\mathrm{FG}}
\DeclareMathOperator{\Men}{\mathcal M^{\heart, = n}_\mathrm{FG}}

\DeclareMathOperator{\G}{\mathbf G}

\DeclareMathOperator{\Mod}{\operatorname{Mod}}

\DeclareMathOperator{\cMod}{\operatorname{cMod}}

\renewcommand{\i}{\infty}

\renewcommand{\Pr}{\mathcal P\mathrm r}
\newcommand{\PrL}{\mathcal P\mathrm r^{\mathrm L}}

\newcommand{\w}{\widehat}

\renewcommand{\i}{\infty}
\renewcommand{\Mn}{\mathcal M^{\mathrm{or}, \le n}_\mathrm{FG}}
\renewcommand{\Mon}{\mathcal M^{\heart, \le n}_\mathrm{FG}}
\renewcommand{\Men}{\mathcal M^{\heart, = n}_\mathrm{FG}}

\DeclareFontFamily{U}{matha}{\hyphenchar\font45}
\DeclareFontShape{U}{matha}{m}{n}{
      <5> <6> <7> <8> <9> <10> gen * matha
      <10.95> matha10 <12> <14.4> <17.28> <20.74> <24.88> matha12
      }{}
\DeclareSymbolFont{matha}{U}{matha}{m}{n}

\DeclareMathSymbol{\varsubset}{3}{matha}{"80}

\usepackage{color}   %May be necessary if you want to color links
\usepackage[hypertexnames=false]{hyperref}
\hypersetup{
    linktoc=all,     %set to all if you want both sections and subsections linked
    linkcolor=black,  %choose some color if you want links to stand out
}

\renewcommand{\i}{\infty}
\renewcommand{\o}{\otimes}

\usepackage{epigraph}

\begin{document}

\begin{abstract}
We define a filtration by open substacks on the non-connective spectral moduli stack of formal oriented groups, which simultaneously encodes and relates the chromatic filtration of spectra and the height stratification of the classical moduli stack of formal groups. Using this open filtration, we express various classical constructions in chromatic  homotopy theory, such as chromatic localization, the monochromatic layer, and $K(n)$-localization, in terms of restriction and completion of sheaves in non-connective spectral algebraic geometry.
\end{abstract}

\maketitle

\section*{Introduction}

This is a follow-up to \cite{ChromaticCartoon}. As set out there, we continue to establish how various aspects of chromatic homotopy theory manifest in terms of non-connective spectral algebraic geometry, using the moduli stack of oriented formal groups $\M$.

In the present work, we employ this perspective to discuss the chromatic filtration. We define a filtration by open  substacks $\Mn$ on $\M$, which recovers the height stratification on the underlying ordinary substacks and the height filtration on sheaves. Through the use of this filtration, we express certain fundamental constructions in contemporary chromatic homotopy theory: chromatic localization, the monochromatic layer, and $K(n)$-localization of spectra, in terms of sheaves in spectral algebraic geometry respectively as: restriction to open substacks, supported sections, and adic completion. In this way, we are realizing the heuristic ideas about the connection between spectral algebraic geometry and chromatic homotopy theory, going back to
 Goerss, Hopkins, Miller, Morava, and others.

\subsection*{Background on the chromatic and height filtrations}
For a fixed prime $p$, the \textit{chromatic filtration} of a spectrum $X$, first introduced in \cite{MRW} for the sphere spectrum and in \cite{Ravenel:Bousfield} in general, is a tower of spectra
\begin{equation}\label{chromatic filtration intro}
\ldots \to L_{n+1}X\to L_nX\to L_{n-1}X\to\cdots \to L_2X\to L_1X,
\end{equation}
converging in many cases (e.g.~for all finite spectra $X$) to the $(p)$-completion $X_{(p)}$. For any $n\ge 1$, the \textit{$n$-th chromatic localization} $L_nX$, while technically defined as a certain Bousfield localization, is closely related to so-called $v_i$-periodicity phenomena for all $i\le n$.

A key reason for the fruitfulness of the chromatic filtration is its deep connection to a certain filtration in a very different setting: the \textit{height stratification}
\begin{equation}\label{height filtration intro}
\mathcal M^{\heart, \le 1}_\mathrm{FG}\subseteq
\mathcal M^{\heart, \le 2}_\mathrm{FG}\subseteq
\cdots\subseteq
\mathcal M^{\heart, \le n-1}_\mathrm{FG}\subseteq
\mathcal M^{\heart, \le n}_\mathrm{FG}\subseteq
\mathcal M^{\heart, \le n+1}_\mathrm{FG}\subseteq
\cdots\subseteq
\mathcal M^\heart_\mathrm{FG}\o_{\mathbf Z}\mathbf Z_{(p)},
\end{equation}
of (the $(p)$-localization of) $\Mo,$ the algebro-geometric stack of formal groups\footnote{The superscript $\heart$ is part of our chosen notation $\Mo$. It serves as a  reminder that, while a stack, this is still an object of usual algebraic geometry, as opposed to say spectral or derived algebraic geometry.}. For any $n\ge 1$, the open substacks $\mathcal M^{\heart, \le n}_\mathrm{FG}\subseteq\Mo\o_{\mathbf Z}\mathbf Z_{(p)}$ is the moduli of formal groups of height $\le n$; see \cite[Section 4]{Smithling}, \cite[Lecture 13]{Lurie Chromatic}. The correspondence between the height stratification and the chromatic filtration relates the open substack $\mathcal M^{\heart, \le n}_\mathrm{FG}$ with the chromatic localization $L_nX$.

The standard approach  of relating formal groups with stable homotopy theory proceeds through Quillen's Theorem on complex bordism $\mathrm{MU}$. Using it,  we may identify the second page of the Adams-Novikov spectral sequence for computing the stable stem with quasi-coherent cohomology as
\begin{equation}\label{ANSS}
E^{s,t}_2 = \mathrm{Ext}^{s,t}_{\mathrm{MU}_*\mathrm{MU}}(\mathrm{MU}_*, \mathrm{MU}_*) \cong \mathrm H^s(\Mo; \omega_{\Mo}^{\o t}).
\end{equation}
This perspective, highlighting the role of the stack of formal groups (as opposed to the Lazard ring, classifying formal group laws) was first emphasized by Hopkins and popularized by  \cite{COCTALOS}.
As explained in accounts on chromatic homotopy theory from this perspective, e.g.~\cite[Section 3.6]{PetersonBook}, \cite[Lecture 21]{Lurie Chromatic} or \cite[Lecture 21]{Pstragowski},  chromatic localizations have to do with  support conditions of quasi-coherent sheaves with respect to the height filtration. Morally, $L_n$ behaves like the restriction along the open inclusion $\mathcal M^{\heart, \le n}_\mathrm{FG}\subseteq\Mo\o_{\mathbf Z}\mathbf Z_{(p)}$.
The goal of this paper is to show how this motto, as well as several others concerning related Bousfield localization relevant to chromatic homotopy theory, can be made perfectly rigorous in the world of spectral algebraic geometry.

\subsection*{Spectral algebraic geometry and the main result}

Spectral algebraic geometry is a homotopical analogue of  algebraic geometry, with the fundamental role of commutative rings replaced by $\mathbb E_\infty$-ring spectra. Extensive foundations of the field are laid in \cite{SAG}, though our needs necessitate a non-connective functor of points approach elaborated on in \cite[Section 1]{ChromaticCartoon}. Briefly, we view non-connective spectral stacks as certain kinds of $\infty$-categorical functors from the $\infty$-category of $\E$-rings $\CAlg$ to the $\i$-category of spaces $\mS$ - a setup convenient for discussing moduli problems. The basic formalism of algebraic geometry goes through in this context, giving rise to notions such as global functions and quasi-coherent sheaves.

The main result of our preceding paper may be summarized as follows:

\begin{theoremm}[{\cite[Theorem 9]{ChromaticCartoon}}]\label{Old Main Theorem}
There exists a canonical  non-connective spectral  stack $\M$, such that:
\begin{enumerate}[label = (\roman*)]
\item Its underlying ordinary stack is the ordinary stack of formal groups $\Mo$.
\item There is a canonical equivalence of $\E$-rings $\sO(\M)\simeq S$ between its global functions and the sphere spectrum.
\item The constant quasi-coherent sheaf functor $X\mapsto X\o\sO_{\M}$ induces an equivalence of $\i$-categories
$$
\Sp\simeq\mathrm{IndCoh}(\M),
$$
under an appropriate  definition of ind-coherent sheaves \cite[Definition 2.4.2]{ChromaticCartoon}.
\item The descent spectral sequence for $\sO_{\M}$ recovers the Adams-Novikov spectral sequence, and the isomorphism \eqref{ANSS}.
\end{enumerate}
\end{theoremm}

The non-connective spectral stack $\M$ is called  the \textit{moduli stack of oriented formal groups}. As the name suggests, it is defined as the moduli stack, parametrizing formal groups in spectral algebraic geometry, equipped with an extra structure of an orientation due to Lurie \cite{survey}, \cite{Elliptic 2}. The preceding theorem may be interpreted as asserting that algebraic geometry of the stack $\M$ corresponds to chromatic homotopy theory.

The goal of the present paper is to strengthen this claim, by showing how the chromatic filtration and the height filtration are simultaneously encoded in the $\M$.  Making this precise, our main theorem is:

\begin{theoremm}[Lemma \ref{Lemma 1-line}, Corollary \ref{L_n as restriction}, Proposition \ref{Chromatic filtration = chromatic filtration}, Theorem \ref{Chromatic localization is restriction}]\label{Main Theorem Intro}
There exists a canonical filtration by open non-connective spectral substacks
$$
\mathcal M^{\mathrm{or}, \le 1}_\mathrm{FG}\subseteq
\mathcal M^{\mathrm{or}, \le 2}_\mathrm{FG}\subseteq
\cdots\subseteq
\mathcal M^{\mathrm{or}, \le n-1}_\mathrm{FG}\subseteq
\mathcal M^{\mathrm{or}, \le n}_\mathrm{FG}\subseteq
\mathcal M^{\mathrm{or}, \le n+1}_\mathrm{FG}\subseteq
\cdots\subseteq
\M\o S_{(p)}
$$
on the $(p)$-localization of the moduli stack of oriented formal groups, such that the following hold:
\begin{enumerate}[label = (\roman*)]
\item \label{point 1 of MTI} Passage to underlying ordinary stacks recovers the usual height stratification \eqref{height filtration intro} on the $(p)$-localization of  the ordinary stack of formal groups $\M\o_{\mathbf Z}\mathbf Z_{(p)}$.

\item \label{point 2 of MTI} For a $(p)$-local spectrum $X$, passage to global sections from the constant quasi-coherent sheaf $X\mapsto X\o\sO_{\M}$ recovers the chromatic filtration \eqref{chromatic filtration intro}. In particular, there is an equivalence of spectra
$$\Gamma(\mathcal M^{\mathrm{or}, \le n}_\mathrm{FG}; \sO_{\M}\o X)\simeq L_nX$$
between the sections of the constant quasi-coherent sheaf over $\mathcal M^{\mathrm{or}, \le n}_\mathrm{FG}$, and the $n$-th chromatic localization.

\item \label{point 3 of MTI}For every fixed finite height $1\le n <\infty$, the functor $X\mapsto X\o\sO_{\M}$ induces an equivalence of $\i$-categories
$$
L_n\Sp\,\simeq \,\QCoh(\mathcal M^{\mathrm{or}, \le n}_\mathrm{FG})
$$
between the $L_n$-local spectra and quasi-coherent sheaves on $\Mn$.
\end{enumerate}
\end{theoremm}

\begin{proof}[Sketch of proof]
For any fixed height $1\le n<\infty$, we define the non-connective spectral substack $\mathcal M_\mathrm{FG}^{\mathrm{or}, \le n}$ as the moduli stack of those  oriented formal groups, whose underlying ordinary formal groups have height $\le n$. In place of  the simplicial description of the moduli stack of oriented formal groups
$$
\xymatrix{
\M\simeq \varinjlim\Big( \cdots  \ar@<-1.5ex>[r]\ar@<-.5ex>[r] \ar@<.5ex>[r] \ar@<1.5ex>[r]& \Spec(\MP\o \MP\o \MP)\ar@<-1ex>[r]\ar[r] \ar@<1ex>[r] &\Spec(\MP\o \MP)\ar@<-.5ex>[r]\ar@<.5ex>[r] & \Spec(\MP)\Big )
}
$$
in terms of the periodic complex bordism spectrum $\MP$, that was instrumental in proving Theorem \ref{Old Main Theorem} in \cite{ChromaticCartoon}, we make use of the simplicial presentation
$$
\xymatrix{
\mathcal M^{\mathrm{or}, \le n}_\mathrm{FG}\simeq \varinjlim\Big( \cdots  \ar@<-1.5ex>[r]\ar@<-.5ex>[r] \ar@<.5ex>[r] \ar@<1.5ex>[r]& \Spec(E_n\o E_n\o E_n)\ar@<-1ex>[r]\ar[r] \ar@<1ex>[r] &\Spec(E_n\o E_n)\ar@<-.5ex>[r]\ar@<.5ex>[r] & \Spec(E_n)\Big )
}
$$
in terms of the Lubin-Tate spectrum $E_n$. Using this presentation, point \ref{point 1 of MTI} amounts to the standard observation that the Hopf algebroid $(\pi_0(E_n), \pi_0(E_n\otimes E_n))$ defines a groupoid presentation of the ordinary stack $\Mon$ of formal groups of height $\le n$.
For point \ref{point 2 of MTI}, the simplicial presentation allows us to recognize 
$$
\Gamma(\Mn; \, X\o \sO_{\M})\simeq \Tot(E_n^{\o \bullet+1}\o X)\simeq X^\wedge_{E_n}
$$
as the $E_n$-nilpotent completion of the spectrum $X$. That this coincides with the chromatic localization $L_nX$ is equivalent to the celebrated Hopkins-Ravenel Smash Product Theorem. The proof of point \ref{point 3 of MTI} is similar:  the $\i$-category of quasi-coherent sheaves on the non-connective spectral stack $\Mn$ may be written as the totalization
$$
\QCoh(\Mn)\simeq \Tot\big( \Mod_{E_n^{\otimes (\bullet +1)}} \big).
$$
Its equivalence to the localized stable $\i$-category $L_n\Sp$ is well-known, e.g.~\cite[Proposition 3.4.1]{Picard}, and follows via standard descent techniques from the fact that the $\E$-ring map $L_nS\to E_n$ is descendable - another consequence of the Smash Product Theorem.
\end{proof}

\begin{remmark}
One key difference between Theorems \ref{Old Main Theorem} and \ref{Main Theorem Intro} is in their respective points \ref{point 3 of MTI}. Obtaining the $\i$-category of spectra $\Sp$ in Theorem \ref{Old Main Theorem}   requires the use of the $\i$-category of $\IndCoh(\M)$, a `renormalized' version of $\QCoh(\M)$. The finite height $n$ version of Theorem \ref{Main Theorem Intro}, on the other hand, features the usual $\i$-category $\QCoh(\mathcal M^{\mathrm{or}, \le n}_\mathrm{FG})$. This suggests that the failure of the canonical functor $\IndCoh(\M)\to\QCoh(\M)$ to be an equivalence of $\i$-categories is closely related to the failure of chromatic convergence for arbitrary spectra. 
\end{remmark}

\subsection*{Outline of the paper}

Let us  outline the contents of this paper. Much of the work is spent  laying the necessary foundations  for working with the non-connective spectral stack  $\mathcal M_\mathrm{FG}^{\mathrm{or}, \le n}$. 
Unwinding its construction, sketched above in the idea of proof of Theorem \ref{Main Theorem Intro}, we may view it as an instance of the following general procedure:

\begin{conns}\label{Compression and decompression 101}
Let $X$ be a non-connective spectral stack $X$ with underlying ordinary stack $X^\heart$. There exists a canonical map $X\to X^\heart\circ \pi_0$, which we call \textit{the compression map}, and which sends a map $\Spec(A)\to X$ into the corresponding map $\Spec(\pi_0(A))\to X^\heart$ for any $\E$-ring $A$. Given a map of ordinary stacks $Y^\heart\to X^\heart$, we define its \textit{decompression} to be a map of non-connective spectral stacks $Y\to X$, defined through the pullback square
$$
\begin{tikzcd}
Y\ar{r} \ar{d} & Y^\heart\circ\pi_0\ar{d}\\
X\ar{r} & X^\heart\circ\pi_0.
\end{tikzcd}
$$
\end{conns}

\begin{remmark}
While being a fully faithful embedding of ordinary stacks into non-connective spectral stacks, the compression map is different from the usual such embedding under which we casually identify ordinary stacks as special instances of spectral ones. That would not be the case for the analogous embeddings in the context of connective spectral algebraic geometry, thanks to the adjunction of $\i$-categories $ \pi_0: \CAlgcn\rightleftarrows \CAlg^\heart  $. But it is the case non-connectively, where the functor $\pi_0: \CAlg\to \CAlg^\heart $ is  no longer an adjoint to the subcategory inclusion $\CAlg^\heart \subseteq \CAlg$ of ordinary commutative rings into $\E$-rings.
\end{remmark}

\subsubsection*{Section 1}
The  decompression procedure of Construction \ref{Compression and decompression 101} in general leads to some rather ill-behaved objects. But in certain special cases, it turns out to be a sensible construction. In Section \ref{Section 1}, we prove this is the case for open immersions $U^\heart\to X^\heart$, leading to a sensible theory of corresponding open substacks $U\subseteq X$ of non-connective spectral stacks. We characterize the $\i$-category of quasi-coherent sheaves on $U$ as the semi-orthogonal complement to the quasi-coherent sheaves on $X$, with homotopy sheaves supported along the reduced closed complement $K :=X^\heart - U^\heart\subseteq X^\heart$.

\subsubsection*{Section 2}
We spend Section \ref{Section 2} with a straightforward goal: to define a good notion of the formal completion of a non-connective spectral stack $X$ along a closed reduced substack $K\subseteq X^\heart.$ 
Under some reasonable conditions, we establish a canonical equivalence
\begin{equation}\label{compactformalI}
\QCoh(X^\wedge_K)\simeq \QCoh_K(X).
\end{equation}
between the $\i$-categories of quasi-coherent sheaves on the formal completion $X^\wedge_K$, and quasi-coherent sheaves on $X$ with support (of homotopy sheaves) in $K$.
This however requires us to work out some basics of formal algebraic geometry in the non-connective setting, based on complete adic $\E$-rings  - $\E$-rings $A$ together with an adic topology on $\pi_0(A)$, with respect to which they are complete. In the ``locally-ringed-space" connective setting, such a theory is developed extensively in \cite[Chapter 8]{SAG}, but the non-connective ``functor of points" introduces a host of new challenges, making it rather tricky.
 One main reason for this is that, in the absence of something like a complete Noetherian local assumption, adic completeness fails to satisfy flat descent. 

To circumvent this difficulty, we spend a large portion of Section \ref{Section 2} juggling  non-complete adic $\E$-rings, which have good descent properties, and complete adic $\E$-rings, which are the real central objects of formal algebraic geometry.  As a paradigmatic example, we define the formal completion $X^\wedge_K$ in two steps. First, we use decompression to define  the \textit{adic pre-completion} $X_K$, as a functor $\CAlg_\mathrm{ad}\to\mS$  from the $\i$-category of adic $\E$-rings. And then, for an arbitrary (pre)stack  $Y:\CAlg_\mathrm{ad}\to \mS$, we define its \textit{completion} $Y^\wedge : \CAlg_\mathrm{ad}^\mathrm{cpl}\to\mS$ to be the restriction to the full subcategory $\CAlg_\mathrm{ad}^\mathrm{cpl}\subseteq\CAlg_\mathrm{ad}$ of complete adic $\E$-rings. Formal completion $X^\wedge_K$ is then  defined to be the completion of the adic pre-completion $X_K$.

\subsubsection*{Section 3}
We finally turn to proving Theorem \ref{Main Theorem Intro} in Section \ref{Section 3}. That largely amounts to  specializing the abstract theory of the previous two sections to the special case of the stack $\M$. Other than proving the statement of Theorem \ref{Main Theorem Intro},  we also provide  spectral-algebro-geometric description of various related notions  in chromatic homotopy theory. Recall, for instance from \cite[Definition 7.5.1]{Orange Book} or \cite[Lectures 32 \& 34]{Lurie Chromatic},
 that the \textit{chromatic acyclization} $C_nX$ and the \textit{monochromatic layer} $M_nX$ of a $(p)$-local spectrum $X$ are defined to sit inside fiber sequences of spectra
$$
C_n X\to X\to L_n X
\qquad
M_nX\to L_nX\to L_{n-1}X. 
$$
They roughly encode all the information about $X$ at chromatic heights $> n$ and $=n$ respectively.
In terms of formal groups, they have to do with the ordinary moduli stacks $\mathcal M^{\heart, \ge n+1}_\mathrm{FG}$ and $\mathcal M^{\heart, = n}_\mathrm{FG}$ of formal groups of height $\ge n+1$ and exact height $n$ respectively.  We give a spectral-algebro-geometric  description of chromatic acyclization and the monochromatic layer as local cohomology, or equivalently sections with specified support.

\begin{theoremm}[Corollary \ref{C_n as local cohomology}, Lemma \ref{monochrom from qcoh intro}]
For any  finite $(p)$-local spectrum $X$, there is a canonical equivalence of spectra
$$
C_nX\,\simeq \,\Gamma_{\mathcal M^{\heart, \ge n+1}_\mathrm{FG}}(\M\o S_{(p)}; \sO_{\M}\o X).
$$
Similarly, for any  $(p)$-local spectrum $X$, there is a canonical equivalence of spectra
$$
M_nX\,\simeq \,\Gamma_{\mathcal M_\mathrm{FG}^{\heart, =n}}(\mathcal M^{\mathrm{or}, \le n}_\mathrm{FG}; \sO_{\M}\o X).
$$
\end{theoremm}

The second of these local cohomology descriptions extends an equivalence on the level of $\i$-categories:

\begin{theoremm}[Proposition \ref{monochrom from qcoh}]\label{MonochTI}
The equivalence of $\i$-categories of Theorem \ref{Main Theorem Intro}, \ref{point 3 of MTI} restrict to an equivalence of subcategories
$$
M_n\mathrm{Sp}
\,\simeq\,
\QCoh_{\mathcal M^{\heart, =n}_\mathrm{FG}}(\mathcal M^{\mathrm{or}, \le n}_\mathrm{FG})
$$
between  the height $n$ monochromatic spectra, and the quasi-coherent sheaves on $\mathcal M^{\mathrm{or}, \le n}_\mathrm{FG}$ with support along the reduced closed substack $\mathcal M_\mathrm{FG}^{\heart, =n}\subseteq \mathcal M^{\heart, \le n}_\mathrm{FG}$.
\end{theoremm}

We also use the Algebraic Chromatic Convergence Theorem of \cite{Goerss} to prove  a chromatic convergence result Theorem \ref{CCT} for quasi-coherent sheaves on $\M\o S_{(p)}$ under a coherence assumption. As a corollary, we recover  the usual Hopkins-Ravenel Chromatic Convergence Theorem for finite spectra.

Using non-connective formal spectral geometry, as developed in Section \ref{Section 2}, we connect the formal completion of $\Mn$ along the reduced closed  ordinary substack $\Men\subseteq\Mon$ with $K(n)$-complete spectra. As in previous cases, we obtain

\begin{theoremm}[Theorem \ref{knloc}]\label{LTI}
The equivalence of $\i$-categories of Theorem \ref{Main Theorem Intro}, \ref{point 3 of MTI} gives rise to the equivalence 
$$
L_{K(n)}\Sp\simeq \QCoh((\Mn)^\wedge_{\Men})
$$
between the $\i$-category of $K(n)$-local spectra and the quasi-coherent sheaves on the formal completion $(\Mn)^\wedge_{\Men}$. In terms of this equivalence of $\i$-categories, the $K(n)$-localization $L_{K(n)}X$ corresponds to the quasi-coherent sheaf $(\sO_{\Mn}\o X)^\wedge_{\Men}$ for any  $(p)$-local spectrum $X$.
\end{theoremm}

In light of Theorems \ref{MonochTI} and \ref{LTI}, the equivalence of $\i$-categories \eqref{compactformalI} between quasi-coherent sheaves with prescribed support and quasi-coherent sheaves on the formal completion, gives rise to an equivalence $M_n\Sp\simeq L_{K(n)}\Sp$ between the $n$-th monochromatic layer and $K(n)$-local spectra, thus recovering another standard result of chromatic homotopy theory, e.g.~ \cite[Theorem 6.19]{Hovey-Strickland} or \cite[Lecture 33, Proposition 12]{Lurie Chromatic}.

Finally, we relate this to our previous work in \cite{DevHop}. Let the map of ordinary stacks $\Spec (\F)\to \Mon$ classify a formal group of exact height $n$ over $\F$. By the work of Lurie  \cite[Theorem  5.1.5,  Remark  6.0.7]{Elliptic 2}, the  Lubin-Tate spectrum $E_n$ may be described through the equivalence of non-connective spectral stacks
$$
\Spf(E_n)\simeq
(\Mn)^\wedge_{\Spec(\F)}.
$$
On the other hand, it is well known, \cite[Theorem 4.3.8]{Smithling}, \cite[Lecture 19, Proposition 1]{Lurie Chromatic}, or \cite[Theorem 17.9]{Pstragowski} that the stack $\Men$ of formal groups of height exactly $n$ is equivalent to the quotient stack $\Spec(\F)/\mathbb G_n$ along the Morava stabilizer group $\mathbb G_n$. By combining all of these facts with some formal properties of the formal completion construction, we obtain a series of equivalences of non-connective formal spectral stacks
\begin{eqnarray*}
(\Mn)^\wedge_{\Men} &\simeq & (\Mn)^\wedge_{\Spec(\F)/\mathbb G_n}\\
&\simeq & \big((\Mn)^\wedge_{\Spec(\F)}\big)/\mathbb G_n\\
&\simeq & \Spf(E_n)/\mathbb G_n.
\end{eqnarray*}
The quotient at the end is taken along the action of the Morava stabilizer group on the non-connective affine formal spectral scheme $\Spf(E_n)$ which is discussed more extensively from this perspective in \cite[Proposition 2.7]{DevHop}. The above equivalence of non-connective formal spectral stacks induces, in light of Theorem \ref{LTI}, on quasi-coherent sheaves the equivalence of $\i$-categories
\begin{equation}\label{eq1I}
L_{K(n)}\Sp\simeq \Mod_{E_n}^{\mathbb G_n}
\end{equation}
between $K(n)$-local spectra and spectral Morava modules, first proved in \cite[Proposition 10.10]{Mathew}. Upon passage to global functions, Theorem \ref{LTI} similarly recovers the equivalence of $\E$-rings
\begin{equation}\label{eq2I}
L_{K(n)}S\simeq E_n^{h\mathbb G_n}
\end{equation}
of the \textit{Devinatz-Hopkins Theorem} of \cite{Devinatz-Hopkins}. Unwinding the above-sketched deductions of the equivalences \eqref{eq1I} and \eqref{eq2I} from Theorem \ref{LTI} reveals them to be nothing but the arguments given to prove these results in \cite[Proofs of Corollary 2.17 and Theorem 2.14]{DevHop}.
It is in this sense that our prior work in the formal setting in \cite{DevHop} is compatible with the ``global" results of this paper and its predecessor \cite{ChromaticCartoon}.

\subsection*{Acknowledgments}
I am grateful to David Ben-Zvi, Andrew Blumberg, and Paul Goerss for  the support, help, encouragement, and advice regarding this project.
Thanks also to Tom Gannon, Sam Raskin, and Saad Slaoui for useful conversations.

\section{Some constructions in non-connective spectral algebraic geometry}\label{Section 1}
For the needs of this paper, we must discuss a few basic constructions in non-connective spectral algebraic geometry. We will adopt the setting and notation  from the functor of points approach laid out in \cite[Chapter 1]{ChromaticCartoon}.
 Unlike the constructions discussed there however, our present needs necessitate focusing on a less familiar construction, which we call \textit{compression}, whose role is not (explicitly) recognized in \cite{SAG}.

\subsection{Compression \& open immersions}

Let $X$ be a non-connective spectral stack. It has an associated underlying ordinary stack $X^\heart$. Under the usual embedding $\Shv^\heart_\mathrm{fpqc}\subseteq\Shv^\mathrm{nc}_\mathrm{fpqc}$, there is no canonical map connecting $X$ and $X^\heart$, like the truncation map $X^\heart\to X$ in the connective case.

 \begin{prop}\label{compressive adjunction}
 The underlying ordinary stack functor $(-)^\heart:\Shv_\mathrm{fpqc}^\mathrm{nc}\to\Shv_\mathrm{fpqc}^\heart$ admits a fully faithful right adjoint $(-) \circ \pi_0 : \Shv_\mathrm{fpqc}^\heart\to\Shv_\mathrm{fpqc}^\mathrm{nc}.$
 \end{prop}
 
 \begin{proof}
 By definition, the functor $X\mapsto X^\heart$ on non-connective stacks factors as
 $$
 \Shv_\mathrm{fpqc}^\mathrm{nc}\xrightarrow{\tau_{\ge 0}}\Shv_\mathrm{fpqc}\xrightarrow{(-)^\heart} \Shv_\mathrm{fpqc}^\heart.
 $$
  To find the right adjoint, and to verify that it is fully faithful, it therefore suffices to do so for each of these functors separately.

Using the notation and results from the proof of  \cite[Proposition 1.2.1]{ChromaticCartoon}, the first of these functors $\tau_{\ge 0} :\Shv^\mathrm{nc}_\mathrm{fpqc}\xrightarrow{\tau_{\ge 0}}\Shv_\mathrm{fpqc}$ is defined to be the left adjoint of the composition functor $\tau_{\ge 0}^* : \Shv_\mathrm{fpqc}\to\Shv_\mathrm{fpqc}^\mathrm{nc},$ which sends $X:\CAlg^\mathrm{cn}\to\mS$ to
 $$
 \CAlg\xrightarrow{\tau_{\ge 0}}\CAlg^\mathrm{cn} \xrightarrow{X}\mS.
 $$
 It is shown in the proof of \cite[Proposition 1.2.1]{ChromaticCartoon} that this functor $\tau_{\ge 0}^*$ is indeed fully faithful.
 
 The second functor $(-)^\heart :\Shv_\mathrm{fpqc}\to\Shv_\mathrm{fpqc}^\heart$ is on the other hand defined to be the left adjoint of the functor $(\pi_0)^* : \Shv_\mathrm{fpqc}^\heart\to\Shv_\mathrm{fpqc}$, which sends $X:\CAlg^\heart\to\mS$ to
 $$
 \CAlg^\mathrm{cn}\xrightarrow{\pi_0}\CAlg^\heart\xrightarrow{X}\mS,
 $$
 and is also fully faithful by the proof of  \cite[Proposition 1.2.1]{ChromaticCartoon}
 In summation, we find that the underlying functor functor $X\mapsto X^\heart$ indeed admits a fully faithful right adjoint given by $(X\circ \pi_0)(A) := X(\pi_0(\tau_{\ge 0}(A))) = X(\pi_0(A))$ for any $\E$-ring $A$. 
 \end{proof}

\begin{definition}
 For any $X\in \Shv_\mathrm{fpqc}^\mathrm{nc}$, we call the unit $X\to X^\heart\circ\pi_0$ of the adjunction furnished by Proposition \ref{compressive adjunction} the \textit{compression of $X$}.
\end{definition}
 
 \begin{remark}
 For a non-connective spectral stack $X$ and  any $\E$-ring $A$, compression gives a canonical map
$$
X(A)\to X^\heart(\pi_0(A))
$$
in the $\i$-category of spaces, sending a map $\Spec(A)\to X$ in $\Shv^\mathrm{nc}_\mathrm{fpqc}$ to its its image $\Spec(\pi_0(A))\to X^\heart$ in $\Shv_\mathrm{fpqc}^\heart$ under the functor $(-)^\heart$. Informally, compression encodes the data encoded in $X$ which is detectable on the underlying set-theoretic level. It serves in certain situations as an analogue of the underlying (Zariski) topological space of a scheme.
 \end{remark}
 
 \begin{remark}
 Let $\M$ be the non-connective spectral stack of oriented formal groups of \cite[Section 2]{ChromaticCartoon}. We know by \cite[Corollary 2.3.9]{ChromaticCartoon} that its underlying ordinary stack is the ordinary stack of formal groups $\mathcal M^\heart_\mathrm{FG}$.
 The compression map $\M\to  \mathcal M_\mathrm{FG}^\heart\circ \pi_0$ then amounts to, for any complex periodic $\E$-ring $A$, sending its Quillen formal group  $\w{\G}{}^{\CMcal Q}_A\simeq \Spf(C^*(\mathbf{CP}^\i; A))$ to the classical Quillen formal group  $\w{\G}{}^{\CMcal Q_0}_A\simeq \Spf(A^*(\mathbf{CP}^\i)).$

 \end{remark}
 
 Though the concept of compression may seem strange and unfamiliar, it actually implicitly features in a number constructions in connective spectral algebraic geometry.

\begin{lemma}\label{compressing open subschemes}
Let  $X$ be a non-connective spectral scheme, and $U\subseteq X$ an open subscheme. The canonical commutative diagram of non-connective spectral stacks induced by compression
$$
\begin{tikzcd}
U\ar{r} \ar{d} & U^\heart\circ\pi_0\ar{d}\\
X\ar{r} & X^\heart\circ\pi_0
\end{tikzcd}
$$
is a pullback square.
\end{lemma}

\begin{proof}
For any $\E$-ring $A$, the map of spaces $(X\times_{(X^\heart\circ \pi_0)}(U^\heart\circ\pi_0))(A)\to X(A)$ is the inclusion of those connected components of $X(A)$ for which the associated $\pi_0(A)$-point of the underlying ordinary scheme $X^\heart$ belongs to the open subscheme $U^\heart$. It follows from the theory of spectral schemes laid out in \cite[Section 1.1]{SAG}, specicically from \cite[Corollary 1.1.6.2]{SAG} (though see also \cite[Proposition 1.6.3.1]{SAG} for a comparison with the functor of points approach) that non-connective spectral schemes may be viewed as spectrally ringed topological spaces, with the underlying ringed topological spaces being the Zariski topological spaces of the underlying ordinary schemes. From that perspective, an open subscheme is determined purely on the topological space, and hence underlying ordinary scheme, level. The same description as we gave above thus characterizes the inclusion $U(A)\to X(A)$, and so the canonical map $U(A)\to (X\times_{(X^\heart\circ \pi_0)}(U^\heart\circ\pi_0))(A)$ is indeed a homotopy equivalence.
\end{proof}

The preceding Lemma shows that in non-connective spectral schemes, open immersions are determined via compression on the level of underlying ordinary schemes. For more general non-connective spectral stacks, we may turn this around into a ``decompression" mechanism for producing open immersions.

\begin{prop}\label{Decompressing opens}
Let $X$ be a non-connective spectral stack, with underlying ordinary stack $X^\heart$. Let $U^\heart\to X^\heart$ be an open immersion of ordinary stacks. Then the upper horizontal arrow in the pullback diagram of non-connective spectral stacks
$$
\begin{tikzcd}
U\ar{r} \ar{d} & U^\heart\circ\pi_0\ar{d}\\
X\ar{r} & X^\heart\circ\pi_0
\end{tikzcd}
$$
is also an open immersion. The map $U\to U^\heart\circ \pi_0$ exhibits an equivalence between the underlying ordinary stack of $U$ and the ordinary stack $U^\heart$ (justifying the notation).
\end{prop}

\begin{proof}
Fix an arbitrary $\E$-ring $A$ and map $\Spec(A)\to X$. To prove that $U\to X$ is an open immersion, we need to show that $U':=\Spec(A)\times_{X^\heart} U^\heart$ is representable by a non-connective spectral scheme, and that the map $U'\to\Spec(A)$ is an open immersion.

By the definition of compression, the morphism $\Spec(A)\to X^\heart$ factors through the compression map $\Spec(A)\to\Spec (A)^\heart\circ \pi_0\simeq \Spec(\pi_0(A))\circ\pi_0$. By assumption, the morphism of ordinary stacks $U^\heart\to X^\heart$ is an open immersion, hence its pullback $\Spec(\pi_0(A))\times_{X^\heart} U^\heart\to \Spec (\pi_0(A))$ is also. It follows that $\Spec(\pi_0(A))\times_X^\heart U^\heart\simeq U'^\heart$ is (representable by) an ordinary scheme, and the induced map $U'^\heart\to  \Spec (\pi_0(A))$ is an open subscheme. Now we may invoke Lemma \ref{compressing open subschemes} to find that the pullback square
$$
\begin{tikzcd}
U'\ar{r} \ar{d} &U'^\heart\circ\pi_0 \ar{d}\\
\Spec(A)\ar{r} & \Spec(\pi_0(A))\circ\pi_0
\end{tikzcd}
$$
exhibits $U'$ as an open non-connective spectral subscheme of $\Spec(A)$. This proves that $U\to X$ is an open immersion. 

Next we wish to show that $U\to U^\heart\circ \pi_0$ induces an equivalence upon underlying ordinary stacks.
Suppose that $X\simeq \varinjlim_i \Spec(A_i)$ for some diagram of $\E$-rings $A_i$. Let us denote $U_i:=U\times_X \Spec(A_i)\to \Spec(A_i)$.
Because we are working in $\Shv_\mathrm{fpqc}^\mathrm{nc}$, which is an $\i$-topos, we find that
$$
U\simeq U\times_X X\simeq U\times_X (\varinjlim_i \Spec(A_i))\simeq \varinjlim_i (U\times_X \Spec(A_i)) = \varinjlim_i U_i.
$$
On the other hand, it follows from what we have already proved previously, that $U_i:=U\times_X \Spec(A_i)\to \Spec(A_i)$ is the open non-connective spectral subscheme defined by the ordinary open  subscheme $U_i^\heart := U^\heart \times_{X^\heart }\Spec(\pi_0(A_i))$ for every $i$. In particular\footnote{Note that this is not automatic, because the underlying ordinary stack functor $X\mapsto X^\heart$ does not in general preserve limits in the  non-connective setting!}, as the notation suggests, $U^\heart_i$ is the underlying ordinary subscheme of $U_i$. Thus the underlying ordinary stack of $U$ is
$
(\varinjlim_i U_i)^\heart=\varinjlim_i U_i^\heart.
$
But  because $X^\heart\simeq \varinjlim_i \Spec(\pi_0(A_i)),$ we may perform an analogous argument as above
$$
U^\heart\simeq U^\heart\times_{X^\heart} X^\heart\simeq U^\heart\times_{X^\heart} (\varinjlim_i \Spec(\pi_0(A_i)))\simeq \varinjlim_i (U^\heart\times_{X^\heart} \Spec(\pi_0(A_i))) = \varinjlim_i U_i^\heart.
$$
Due to naturality of all these equivalences, we may conclude that the compression map $U\to U^\heart\circ\pi_0$ indeed exhibits an equivalence upon the underlying ordinary stack $U^\heart.$
\end{proof}

\begin{remark}\label{connective cover and decompression}
Let $c : X\to \tau_{\ge 0}(X)$ be the connective cover of \cite[Proposition 1.2.1]{ChromaticCartoon}. By definition of underlying ordinary stacks, the connective cover map $c$ induces an equivalence $c^\heart :X^\heart\simeq (\tau_{\ge 0}(X))^\heart$ upon them. It follows that we have, in the setting of Proposition \ref{Decompressing opens}, the following diagram of non-connective spectral stacks
$$
\begin{tikzcd}
U\ar{r}{c} \ar{d} & \tau_{\ge 0}(U)\ar{r}\ar{d} & U^\heart\circ\pi_0\ar{d}\\
X\ar{r}{c} & \tau_{\ge 0}(X)\ar{r} & X^\heart\circ\pi_0,
\end{tikzcd}
$$
in which all three squares are pullback squares. Focusing just on the left square, we find a slightly different description of decompression. If we already know that an open ordinary substack $U^\heart\subseteq X^\heart$ induces an open spectral substack $\tau_{\ge 0}(U)\subseteq\tau_{\ge 0}(X)$, then the open non-connective spectral substack $U\subseteq X$ is obtained by pullback along the connective cover map $c:X\to\tau_{\ge 0}(X)$. The reason this is different is that in the connective world, there exists a map $X^\heart\to \tau_{\ge 0}(X)$, which may be viewed as exhibiting the correspondence between the ordinary open substacks $X^\heart$ and the open spectral substacks of $\tau_{\ge 0}(X)$. This eliminates the slightly strange functor $X^\heart\circ \pi_0$ from the discussion of decompression, but at the point of it being a two-step construction of push-pulling along the cospan $X^\heart\to \tau_{\ge 0}(X)\leftarrow X$.
\end{remark}

One could in principle try to perform the decompression construction as in Proposition \ref{Decompressing opens} for any map of ordinary stacks $Y^\heart\to X^\heart$. However, the decompressed stack will in general not inherit any niceness properties of $Y^\heart$, e.g.\ geometricity. The only three cases that we are aware decompression to be viable in, is the open immersion situation of Proposition \ref{Decompressing opens}, the much more involvedly, formal completions in the sense of Definition \ref{completion def}, and the \'etale base-change described in the next Remark.

\begin{remark}\label{remark about etale base-change}
Let $X\to\Spec(A)$ be a map of non-connective spectral stacks, and let $A\to B$ be an \'etale map of $\E$-rings.  We denote the base-change of $X$ along this map by $X\o_A B:= X\times_{\Spec(A)}\Spec(B)$.
The standard result \cite[Theorem 7.5.0.6]{HA} that the map $B\mapsto \pi_0(B)$ induces an equivalence of $\i$-categories $\CAlg_A^{\text{\'et}}\simeq \CAlg_{\pi_0(A)}^{\heart, \text{\'et}}$ implies by a similar argument to the proof of Proposition \ref{Decompressing opens} that all the three squares in the diagram of non-connective spectral stacks
$$
\begin{tikzcd}
X\o_A B\ar{r} \ar{d} &
\ar{r}\ar{d} 
(X^\heart\o_{\pi_0(A)} \pi_0(B))\circ \pi_0
&
\Spec(\pi_0(B))\circ\pi_0 \ar{d}\\
X\ar{r} &
X^\heart\ar{r}\circ \pi_0&
 \Spec(\pi_0(A))\circ\pi_0
\end{tikzcd}
$$
are pullback squares. Consequently, the \'etale base-change on the level of non-connective stacks $X\o_A B$ is obtained by decompression from the \'etale base-change of the underlying ordinary stacks $X^\heart\o_{\pi_0(A)}\pi_0(B)$.
\end{remark}

\subsection{Quasi-coherent sheaves supported on a reduced closed substack} This subsection closely follows \cite[Subsections 7.1.5 \& 7.2.3]{SAG}, working out an analogue of those results in our setting.

\begin{remark}\label{etale decompression}
In what follows, we nonchalantly discuss reduced closed substacks of geometric ordinary stacks. Due to the notion of reducedness really only satisfying smooth (and consequently \'etale and Zariski) descent, some care must be taken here.  We refer to \cite[Section 5.2]{Goerss} for a discussion of reduced closed substacks - note that the notion of algebraic stacks in \textit{loc.~cit.} is a version (restricting functor values to $\mathrm{Grpd}\simeq \mS^{\le 1} \subseteq \mS$) of our notion of a geometric ordinary stack.
\end{remark}

Let $X$ be a non-connective spectral stack. Let $U^\heart\subseteq X^\heart$ be a quasi-compact open substack, with complementary reduced closed substack $K = X^\heart - U^\heart.$ Let $j : U\to X$ be the corresponding open immersion from Proposition \ref{Decompressing opens}. Assume that both $X$ and $U$ are geometric, in the sense of \cite[Definition 1.3.1]{ChromaticCartoon}. These assumptions will remain in force until the end of this subsection.

\begin{definition}
A quasi-coherent sheaf $\sF$ on $X$ is \textit{supported on $K$} if $j^*(\sF) \simeq 0$. Let  $\QCoh_K(X)\subseteq\QCoh(X)$ denote the full subcategory spanned by all quasi-coherent sheaves supported on $K$.
\end{definition}

Because the functor $j^*$ preserves colimits, the subcategory $\QCoh_K(X)\subseteq\QCoh(X)$ is closed under colimits. This colimit-preservation of the subcategory inclusion,  and presentability of the $\i$-categories in question, combine with the Adjoint Functor Theorem to guarantee the existence of a right adjoint $\Gamma_K : \QCoh(X)\to\QCoh_K(X)$.

\begin{remark}
For a quasi-coherent sheaf $\sF\in\QCoh(X)$, the sheaf $\Gamma_K(\sF)$ is to be thought of as the sheaf of those sections of $\sF$ which are supported along $K\subseteq X$. In particular, $\Gamma_K(X; \sF) := \Gamma(X; \Gamma_K(\sF))$ is the of global sections of $\sF$ with support in $K$.
\end{remark}

\begin{remark}\label{remark affine setting for nil}
Let $X=\Spec (A)$ for an $\E$-ring $A$, and $K\subseteq\Spec (\pi_0(A))$ is the closed subscheme determined by a finitely generated ideal $I\subseteq\pi_0(A)$. Then by\footnote{The statement in \textit{loc.~cit.} is restricted to the connective setting, but the proof does not require it. This is true of many parts of \cite[Section 7.1.5]{SAG}, see \cite[Remark 7.1.1.11]{SAG} for a partial explanation.} \cite[Proposition  7.1.5.3]{SAG}, the canonical equivalence of $\i$-categories $\QCoh(X)\simeq\Mod_A$ restricts to an equivalence 
$$\QCoh_K(X)\simeq \Mod_A^{\mathrm{Nil}(I)}$$
with the full subcategory of $A$-modules which are $I$-nilpotent in the sense of \cite[Definition 7.1.1.6 or Example 7.1.1.7]{SAG}. That is to say, an $A$-module $M$ is $I$-nilpotent if and only if for every $a\in I$ and every $x\in\pi_*(M),$ we have $a^n x = 0$ for $n>>0.$
\end{remark}

Our goal in the rest of this subsection is to connect quasi-coherent sheaves supported on $K$ with quasi-coherent sheaves on (the decompression of) its open complement $U$.  For this purpose, we will make use of the theory of  semi-orthogonal decompositions, following \cite[Section 7.2]{SAG}.

\begin{lemma}\label{ff for oim}
The quasi-coherent pushforward functor $j_* : \QCoh(U)\to\QCoh(X)$ is fully faithful.
\end{lemma}

\begin{proof}
A right  adjoint functor is fully faithful if and only if the counit of the adjunction is an equivalence. Therefore we must show that for an arbitrary $\sF\in\QCoh(U)$, the canonical map $j^*j_*(\sF)\to \sF$ is an equivalence.
Consider the pullback square
$$
\begin{tikzcd}
U^\heart\ar{r}{\id} \ar{d}{\id} & U^\heart\ar{d}{j^\heart}\\
U^\heart\ar{r}{j^\heart} & X^\heart
\end{tikzcd}
$$
of ordinary stacks. By applying the compressive embedding $(-)\circ\pi_0:\Shv_\mathrm{fpqc}^\heart\to\Shv_\mathrm{fpqc}^\mathrm{nc}$ (which is a right adjoint, and therefore preserves limits, by Proposition \ref{compressive adjunction}) and pulling back along the compression map $X\to X^\heart\circ \pi_0$, we find that
$$
\begin{tikzcd}
U\ar{r}{\id} \ar{d}{\id} & U\ar{d}{j}\\
U\ar{r}{j} & X
\end{tikzcd}
$$
is a pullback square as well. Thanks to the geometricity and quasi-compactness assumptions, we may invoke the Beck-Chevalley base-change formula
$$
j^*j_*(\sF)\simeq  \id^*\id_*(\sF) \simeq \sF
$$
to conclude that $j_*$ is fully faithful.
\end{proof}

We will sometimes use the fully faithful embedding of Lemma \ref{ff for oim} to identify $\QCoh(U)$ with a full subcategory of $\QCoh(X)$. In that case, we might suppress the functor $j_*$ from notation, and denote its adjoint by $j^*(\sF)=\sF\vert_U$. When the fully faithful embedding is suppressed, $\sF\vert_U$ may stand  for $j_*j^*(\sF)$ as well. The following is a direct analogue of \cite[Proposition 7.2.3.1]{SAG} in the context of non-connective geometric spectral stacks.

\begin{prop}\label{semi-orthog decomp}
The quasi-coherent pushforward $j_* : \QCoh(U)\to\QCoh(X)$ exhibits a semi-direct decomposition $(\QCoh_K(X), \QCoh(U))$  of the stable $\i$-category $\QCoh(X)$. In particular, there is a canonical equivalence of $\i$-categories
$$\QCoh(U)\simeq \QCoh_K(X)^\perp.$$
\end{prop}

\begin{proof}
The functor $j_* : \QCoh(U)\to\QCoh(X)$ has a left adjoint in the quasi-coherent pullback $j^*$. Hence, using the characterization of semi-direct decompositions \cite[Corollary 7.2.1.5]{SAG}, it follows from Lemma \ref{ff for oim} that it  induces a semi-direct decomposition $({}^\perp\QCoh(U),\QCoh(U))$ of $\QCoh(X)$.

We wish to identify the left orthogonal as $\QCoh_K(X)\simeq {}^\perp\QCoh(U).$ To do so, note that we have for any
$\sF\in\QCoh(X)$ and $\sF'\in\QCoh(U)$ a homotopy equivalence
$$
\Map_{\QCoh(U)}(j^*(\sF), \sF')\simeq \Map_{\QCoh(X)}(\sF, j_*(\sF')).
$$
Consequently $\sF\in\QCoh_K(X)$ if and only if $\sF\in{}^\perp\QCoh(U)$.

The final piece of the puzzle is \cite[Corollary 7.2.1.8]{SAG}, which shows that the canonical inclusion $\QCoh(U)\subseteq ({}^\perp\QCoh(U))^\perp$ is in fact an equivalence of $\i$-categories.
\end{proof}

\begin{corollary}\label{corollary semi-orthogonal sequence}
For any quasi-coherent sheaf $\sF$ on $X$, the counit and unit of the relevant adjunctions form a cofiber sequence
$$
\Gamma_K(\sF)\to\sF\to \sF\vert_U
$$
in the stable $\i$-category $\QCoh(X)$.
\end{corollary}

\begin{remark}\label{semi-orth loc-nilp}
Let us specialize again to the case of   Remark \ref{remark affine setting for nil}, i.e.~when $X=\Spec(A)$, $K$ is cut out by a (reduction of a) finitely generated ideal $I\subseteq\pi_0(A)$, and $U$ is the open complement of $K$.
In \cite[Definition 7.2.4.1]{SAG}, the full   subcategory of $I$-local modules $\Mod_A^{\mathrm{Loc}(I)}\subseteq\Mod_A$ is defined as the right orthogonal $(\Mod_A^{\mathrm{Nil}(I)})^\perp$. It therefore follows from Lemma \ref{ff for oim} that the equivalence of $\i$-categories $\QCoh(X)\simeq \Mod_A$, given by global sections $\sF\mapsto \Gamma(X; \sF)$, restricts to an equivalence of subcategories
$$
\QCoh(U)\simeq \Mod_A^{\mathrm{Loc}(I)}.
$$
Under this, the semi-orthogonal decomposition of Proposition \ref{semi-orthog decomp} identifies with the nilpotent-local semi-orthogonal decomposition $(\Mod_A^{\mathrm{Nil}(I)}, \Mod_A^{\mathrm{Loc}(I)})$ of the $\i$-category $\Mod_A$, see \cite[Proposition 7.2.4.4]{SAG}.
The functor $j_*j^*$ is then the localization functor $L_I :\Mod_A\to \Mod_A^{\mathrm{Loc}(I)},$ and the cofiber sequence of Corollary \ref{corollary semi-orthogonal sequence} takes the form
$$
\Gamma_I(M)\to M\to L_I(M)
$$
for any $A$-module $M$.
\end{remark}

The cofiber sequence of Corollary \ref{corollary semi-orthogonal sequence} is very useful for proving the compatibility of the functor $\Gamma_K$ of supported sections with various quasi-coherent pullbacks and pushforwards.

\begin{prop}\label{support and qc pullback}
Let $f:X'\to X$ be a morphism of geometric non-connective spectral stacks, and denote $K' := (X')^\heart\times_{X^\heart} K$. For any quasi-coherent sheaf $\sF\in\QCoh(X)$, there is a canonical equivalence $f^*(\Gamma_K(\sF))\simeq\Gamma_{K'}(f^*(\sF)).$
\end{prop}

\begin{proof}Consider the pullback square of non-connective spectral stacks
$$
\begin{tikzcd}
U'\ar{r}{j'} \ar{d}{f'} & X'\ar{d}{f}\\
U\ar{r}{j} & X.
\end{tikzcd}
$$
The map $j':U'\to X'$ is an open immersion, corresponding on underlying ordinary stacks to the open complement of the reduced closed substack $K'\subseteq (X')^\heart.$ By the base-change formula, applicable in this situation by a variant of \cite[Proposition 6.3.4.1]{SAG}, the canonical map $f^*j_*\to (j')_* (f')^* $ is an equivalence. Using this, and commutativity of the above diagram, we obtain an equivalence
$$
f^*j_*j^*(\sF)\simeq  (j')_* (f')^*j^*(\sF)\simeq  (j')_*(j')^*f^*(\sF)
$$
in $\QCoh(X')$. Taking quasi-coherent pullback along $f$ of the fiber sequence of Corollary \ref{corollary semi-orthogonal sequence}, produces a fiber sequence
$$
f^*(\Gamma_K(\sF))\to f^*(\sF)\to f^*j_*j^*(\sF)
$$
in the stable $\i$-category $\QCoh(X')$.  The right-most term of this sequence may by the above be identified with $ (j')_*(j')^*f^*(\sF)$, and by virtue of naturality of everything so far,  the second morphism of the above fiber sequence can be seen to come from the unit natural transformation $\id\to (j')_* (j')^*$. That is also the second morphism in the fiber sequence
$$
\Gamma_{K'}(f^*(\sF))\to f^*(\sF)\to  (j')_*(j')^*(f^*(\sF)),
$$
obtained by another application of  Corollary \ref{corollary semi-orthogonal sequence}. Comparing the two fiber sequences thus gives rise to the desired equivalence.
\end{proof}

\begin{prop}\label{support and connective cover}
Let $c : X\to \tau_{\ge 0}(X)$ be the connective cover of \cite[Proposition 1.2.1]{ChromaticCartoon}. The canonical map $\Gamma_K(c_*(\sF))\to c_*(\Gamma_K(\sF))$ is an equivalence for any quasi-coherent sheaf $\sF\in\QCoh(X)$.
\end{prop}

\begin{proof}
Recall from Remark \ref{connective cover and decompression} that connective covers induce a pullback square
$$
\begin{tikzcd}
U\ar{r}{c'} \ar{d}{j} & \tau_{\ge 0}(U)\ar{d}{\tau_{\ge 0}} \\
X\ar{r}{c} & \tau_{\ge 0}(X).
\end{tikzcd}
$$
The desired identification between supported sections now follows just like in the proof of Proposition \ref{support and qc pullback}, from the cofiber sequences of Corollary \ref{corollary semi-orthogonal sequence} and the natural equivalences
$$
c_* j_*j^*(\sF)\simeq \tau_{\ge 0}(j)_*(c')_*j^*(\sF)\simeq \tau_{\ge0}(j)_*\tau_{\ge 0}(j)^*c_*(\sF),
$$
obtained from the above pullback square by using its commutativity and the base-change formula respectively.
\end{proof}
 
\begin{remark}
 By \cite[Proposition 1.5.5]{ChromaticCartoon}, pushforward along $c:X\to\tau_{\ge 0}(X)$ loses no information, so long as we also keep track of the $c_*(\sO_X)$-module structure. This allows us to view $c_* :\QCoh(X)\to \QCoh(\tau_{\ge 0}(X))$ as a forgetful functor, in which case Proposition \ref{support and connective cover} asserts that supported sections are insensitive to non-connective enhancements $X$ of a spectral stack $\tau_{\ge 0}(X)$.
 \end{remark}
 
\begin{prop}\label{support and truncation}
Let $t : \tau_{\ge 0}(X)\to X^\heart$ be the truncation map of \cite[Proposition 1.2.1]{ChromaticCartoon}. The canonical map $\Gamma_K(t_*(\sF))\simeq t_*(\Gamma_K(\sF))$ is an equivalence for any quasi-coherent sheaf $\sF\in\QCoh(X^\heart)$.
\end{prop}

\begin{proof}
Repeat the proof of Proposition \ref{support and connective cover}, using the pullback square
$$
\begin{tikzcd}
U^\heart\ar{r}{t'} \ar{d}{j'^\heart} & \tau_{\ge 0}(U)\ar{d}{j'} \\
X^\heart\ar{r}{t} & \tau_{\ge 0}(X).
\end{tikzcd}
$$
That this is indeed a pullback square follows for instance from Proposition \ref{Decompressing opens}, since the right vertical arrow is obtained by decompression.
\end{proof}

\begin{remark}
Proposition \ref{support and truncation} may be viewed as saying that supported sections of quasi-coherent sheaves on an ordinary geometric stack $X^\heart$ are insensitive to spectral enhancements $\tau_{\ge 0}(X)$ of $X^\heart$. 
\end{remark}

\begin{remark}\label{Gamma_K for heart}
For $X^\heart$ itself, the inclusion $\QCoh(X^\heart)^\heart\to\QCoh(X^\heart)$ induces an equivalence of $\i$-categories
$$\QCoh(X^\heart)\simeq\widehat{\mathcal D}(\QCoh(X^\heart)^\heart),$$
between the stable $\i$-category of spectral quasi-coherent sheaves on $X^\heart$ and the completed derived $\i$-category of the abelian category of ordinary quasi-coherent sheaves on $X^\heart$, in the sense of \cite[Definition C.5.9.4]{SAG}. See \cite[Corollary 10.3.4.13]{SAG} for a proof under the additional assumption that $X^\heart$ is spectral Deligne-Mumford. Under that equivalence, and using \cite[Example 7.1.4.6]{SAG} for the affine case, we may deduce that $\Gamma_K(\sF)$ becomes identified with the total right derived functor $R\Gamma_K^\heart(\sF)$ of the ordinary $K$-supported sections functor $\Gamma_K^\heart:\QCoh(X^\heart)^\heart\to\QCoh(X^\heart)^\heart.$
\end{remark}

The following spectral sequence, a global version of \cite[Equation (3.2)]{Greenlees-May}, is a useful  tool for computing compactly supported sections. Here we let $\mathcal H^i_K$ denote the $i$-th derived functor of the ordinary $K$-supported sections functor $\Gamma_K^\heart:\QCoh(X^\heart)^\heart\to\QCoh(X^\heart)^\heart.$

\begin{prop}[Local cohomology spectral sequence]\label{lc spectral sequence}
For any quasi-coherent sheaf $\sF\in\QCoh(X)$, there exists a canonical Adams-graded spectral sequence
$$
E^{s, t}_2 = \mathcal H^s_K(\pi_t(\sF))\Rightarrow \pi_{t-s}(\Gamma_K(\sF))
$$
in the abelian category $\QCoh(X^\heart)^\heart$.
\end{prop}

\begin{proof}
Fix a quasi-coherent sheaf $\sF\in\QCoh(X)$. We may imitate Construction \cite[Construction 1.5.6]{ChromaticCartoon}, using the filtered object
$$
\mathbf Z\ni n \mapsto\Gamma_K(\tau_{\ge  -n}(\sF))\in\QCoh(\tau_{\ge 0}(X))
$$
to obtain a spectral sequence
\begin{equation}\label{partial ss}
E^{s,t}_2 = \pi_{-s}(\Gamma_K(\pi_t(\sF))\Rightarrow \pi_{t-s}(\Gamma_K(\sF))
\end{equation}
in the abelian category $\QCoh(\tau_{\ge 0}(X))^\heart$. It follows from Proposition \ref{support and connective cover} that the $\i$-page is the same regardless of whether $\Gamma_K(\sF)$ is interpreted in $\QCoh(X)$, or whether it is (by an anonymous application of $c_*$) interpreted in $\QCoh(\tau_{\ge 0}(X))$. Using Remark \ref{Gamma_K for heart}, we have for any $\sG\in\QCoh(X^\heart)^\heart$ and any $s\ge 0$ canonical isomorphisms
$$
\pi_{-s}(\Gamma_K(\sG))\simeq \pi_{-s}(R\Gamma_K^\heart(\sG)) \simeq R^s \Gamma_K^\heart(\sG)\simeq \mathcal H^s_K(\sG)
$$
in the abelian category $\QCoh(X^\heart)^\heart$.
Here by Proposition \ref{support and truncation}, the supported sections on the left-most side  may be interpreted either in $\QCoh(\tau_{\ge 0}(X))$ or (by an anonymous application of $t_*$) in $\QCoh(X^\heart)$. Since any homotopy sheaf $\pi_t(\sF)$ by definition always belongs to $\QCoh(\tau_{\ge 0}(X))^\heart\simeq \QCoh(X^\heart)^\heart$, taking $\sG = \pi_t(\sF)$ now identifies the second page of the spectral sequence \eqref{partial ss} with the desired form from the statement of the Proposition.
\end{proof}

\section{Formal completions in non-connective spectral algebraic geometry}\label{Section 2}

We would next like to discuss the process of formally completing a non-connective spectral stack $X$ along a closed (reduced) substack $K\subseteq X^\heart.$  Sheaves on such a formal completion $X^\wedge_K$ should correspond to sheaves on $X$ with support in $K$, as discussed in the previous subsection. The goal of this section is therefore to define $X^\wedge_K$, quasi-coherent sheaves on it, and, under some restrictions on $X$ an $K$, find a canonical equivalence of $\i$-categories
$$
\QCoh(X^\wedge_K)\simeq \QCoh_K(X).
$$ 
 As it turns out, the setup for formal completion in the  world of non-connective spectral algebraic geometry, appropriate for our needs, is a little more intricate than one might hope . See the following remarks for some explanation and justification of these difficulties.

\begin{remark}\label{indispensably adic}
Any connective complete adic $\E$-ring may obtained as a filtered limit of connective $\E$-rings, equipped with the discrete adic topology by \cite[Remark 8.1.2.4]{SAG}. This implies \cite[Theorem 8.1.5.1]{SAG} that connective formal spectral algebraic geometry is fully determined, in the functor of points approach, by restriction to the full subcategory $\CAlg^\mathrm{cn}\subseteq (\CAlg_\mathrm{ad}^\mathrm{cpl})^\mathrm{cn}$. The proof of this however crucially employs induction up the Postnikov tower of truncations - a proof technique which lacks the starting base-case in the non-connective setting. Indeed, we do not know if, or under what conditions, the analogous result holds in non-connective spectral algebraic geometry. This is why in our setting, unlike in  \cite[Definition 18.2.1.1]{SAG}, we are not able to restrict our attention in defining the formal completion to (presheaves on) the subcategory $\CAlg\subseteq\CAlg_\mathrm{ad}^\mathrm{cpl}$ of adic $\E$-rings with discrete topology, equivalent to the usual $\i$-category of $\E$-rings. In the non-connective world, adic rings are indispensable.
\end{remark}

\begin{remark}
Since our interest ultimately lies in complete adic $\E$-rings (the basic affine objects of formal spectral algebraic geometry), it may seem unclear why we would not impose the completeness assumption from the start, yet we instead work with arbitrary adic $\E$-rings for much of this section. The reason  is that we will at certain points, most notably in Proposition \ref{informal charts}, have to rely crucially on the ability to discuss fpqc descent. While  there are no difficulties in porting the fpqc topology to the $\i$-category $\CAlg_\mathrm{ad}$, issues do arise when trying to restrict it to the full subcategory $\CAlg_\mathrm{ad}^\mathrm{cpl}\subseteq\CAlg_\mathrm{ad}$. One of the basic assumptions for a Grothendieck topology is that coverings are stable under base-change. If $A\to B$ is a map of complete adic $\E$-rings, and $A\to A'$ is a flat cover, then the base-change $B\to A'\o_A B$ is indeed also faithfully flat - hence all is good in $\CAlg_\mathrm{ad}$. But to form base-change in $\CAlg_\mathrm{ad}^\mathrm{cpl}$, we must form the completion $A'\widehat{\o}_A B$ of the usual relative smash product. The issue is that the completion map $A'\o_A B\to A'\widehat{\o}_A B$, or more generally $A\to A^\wedge_I$ for any adic $\E$-ring $A$ with an ideal of definition $I\subseteq\pi_0(A)$, is in general not  faithfully flat.
By \cite[\href{https://stacks.math.columbia.edu/tag/00MC}{Tag 00MC}]{stacks-project}, this issue can be overcome by restricting everywhere to adic $\E$-rings $A$ for which $\pi_0(A)$ is a local Noetherian ring with its usual adic topology - that is the approach we followed in \cite{DevHop}. But in the general case, it seems to be important to split up the story into an adic setting, which interacts well with descent, and a smaller formal setting, which does not. As we will see in Proportion \ref{QCoh don't care about sheafification} and Lemma \ref{QCoh don't care about completion} however, at least quasi-coherent sheaves (as we define them) are insensitive to either sheafification or completion.
\end{remark}

\subsection{Adic $\E$-rings} We begin with the discussion of auxiliary notions, required for the discussion of formal completion in our setting.

\begin{definition}\label{def of adic}
An \textit{adic $\E$-ring} is an $\E$-ring $A$, together with an adic topology on $\pi_0(A)$, that is to say, an $I$-adic topology for some finitely generated ideal $I\subseteq\pi_0(A)$. A \textit{map of adic $\E$-rings} is an $\E$-ring map $A\to B$, such that the induced ring homomorphism $\pi_0(A)\to\pi_0(B)$ is continuous with respect to the adic topologies.
\end{definition}

\begin{remark}
The ideal $I\subseteq\pi_0(A)$ in Definition \ref{def of adic} is not part of the data of an adic $\E$-ring $A$. Indeed,  the ideal $I^n$ will do just as well, as might others. On the other hand, the topology on $\pi_0(A)$ is itself part of the data, and any finitely generated ideal $I\subseteq\pi_0(A)$, for which this is the $I$-adic topology, is called an \textit{ideal of definition}. The thing that all ideals of definition share is the reduced closed subscheme $V(I)_\mathrm{red} = V(\sqrt I)\subseteq\Spec(\pi_0(A))$ that they determine, see \cite[Remark 8.1.1.4]{SAG}.
\end{remark}

\begin{remark}
Let $\CAlg_\mathrm{ad}$ denote the $\i$-category of adic $\E$-rings. It follows from the definition that the diagram of $\i$-categories
$$
\begin{tikzcd}
\CAlg_\mathrm{ad}\ar{r}{} \ar{d}{\pi_0} & \CAlg\ar{d}{\pi_0}\\
\CAlg^\heart_\mathrm{ad}\ar{r}{} & \CAlg^\heart,
\end{tikzcd}
$$
in which the unlabeled arrows are ``forgetful" functors, discarding the adic topology, is a pullback square in $\Cat$. We might say that adic $\E$-rings are obtained from  ordinary adic commutative rings  by ``decompression".
\end{remark}

\begin{prop}\label{triv and disc}
The functor $\CAlg_\mathrm{ad}\to\CAlg$, which discards the adic topology, admits both a left and a right fully faithful adjoint. Hence it preserves both limits and colimits.
\end{prop}
\begin{proof}
It suffices to exhibit the adjoints:
\begin{itemize}
\item The left adjoint $\delta:\CAlg\to \CAlg_\mathrm{ad}$ views any $\E$-ring $A$ as an adic $\E$-ring by equipping $\pi_0(A)$ with the discrete topology. This is indeed an adic topology, since its ideals of definition are precisely all nilpotent ideals, e.g.\ $I = (0)$.

\item The right adjoint $\tau:\CAlg\to \CAlg_\mathrm{ad}$ views any $\E$-ring $A$ as an adic $\E$-ring by equipping $\pi_0(A)$ with the trivial topology. This is indeed an adic topology, since its ideal of definition $I = \pi_0(A)$ is generated by any invertible element , e.g.\ $A=(1)$.
\end{itemize}
That these are indeed the respective adjoints is a matter of basic point-set topology. To see that $\delta$ and $\tau$  are fully faithful, we must show that the respective unit and counit of the adjunctions they participate in, are equivalences. That follows from observing that both of the canonical maps $A\to\delta(A)$ and $\tau(A)\to A$ in $\CAlg_\mathrm{ad}$ induce equivalences on underlying $\E$-rings for any adic $\E$-ring $A$.
\end{proof}

The functor $\tau$ from the proof of Proposition \ref{triv and disc} will play no role in the remainder of this paper. The functor $\delta : \CAlg\to\CAlg_\mathrm{ad}$ on the other hand will. We view it as the default embedding of $\E$-rings into adic $\E$-rings, and will sometimes omit it from notation.

The adjunction $\delta :\CAlg\rightleftarrows \CAlg_\mathrm{ad} : U$ induces adjunctions between copresheaf $\i$-categories
$$
\Fun(\CAlg, \mS)
\underset{\delta^*}{\overset{\delta_!}\rightleftarrows}
\Fun(\CAlg_\mathrm{ad}, \mS)
\underset{U^*}{\overset{U_!}\rightleftarrows}
\Fun(\CAlg, \mS)
$$
via the usual left Kan extension and composition. These two adjunctions are connected as
$$
U_!\dashv U^* \simeq \delta_! \dashv \delta^*,
$$
where the equivalence in the middle stems from the adjunction $\delta\dashv U$ by abstract nonsense.

\begin{remark}\label{ff functors on adic}
As consequence of the fact that $U\circ \delta\simeq \mathrm{id}$, we obtain have $\delta^*\circ U^*\simeq (U\circ \delta)^*\simeq \mathrm{id}$. That shows the middle functor $U^*\simeq\delta_!:\Fun(\CAlg, \mS)\to \Fun(\CAlg_\mathrm{ad}, \mS)$ to be fully faithful. We view it as an implicit inclusion of functors $\CAlg\to\mS$ into functors $\CAlg_\mathrm{ad}\to \mS$, and will as such often omit it from notation.
\end{remark}

As discussed in the preceding remark, we will view $\Fun(\CAlg_\mathrm{ad}, \mS)$ as an enlargement of $\Fun(\CAlg, \mS)$. The main thesis of this paper and its predecessors is that the latter is a good setting for a certain portion of non-connective spectral algebraic geometry. As such, it seems reasonable to seek out analogues of some of the basic notions of (non-connective) spectral algebraic geometry in the adic setting as well. The following two definitions are a first step in this program.

\begin{definition}
The \textit{adic spectrum} $\mathrm{Spad}(A):\CAlg_\mathrm{ad}\to\mS$ of an adic $\E$-ring $A$ is the corresponding corepresentable functor $B\mapsto \Map_{\CAlg_\mathrm{ad}}(A, B)$.
\end{definition}

\begin{definition}
A map of adic $\E$-rings is \textit{faithfully flat} if and only if  its underlying $\E$-ring map is faithfully flat in the sense of \cite[Definition B.6.1.1]{SAG}.
\end{definition}

\begin{prop}
Defining coverings in $\CAlg_\mathrm{ad}$ by demanding that the forgetful functor $\CAlg_\mathrm{ad}\to \CAlg$ take them to fpqc coverings in $\CAlg$ in the sense of \cite[Proposition B.6.1.3]{SAG}, defines a Grothendieck topology on the opposite of the $\i$-category $\CAlg_\mathrm{ad}$
\end{prop}

\begin{proof}
We may directly repeat the proof of \cite[Proposition B.6.1.3]{SAG} from the non-adic setting, since its application of \cite[Proposition A.3.2.1]{SAG} requires checking conditions about adic $\E$-rings which follow from the corresponding ones about $\E$-rings by the colimit and limit preservation of Proposition \ref{triv and disc}.
\end{proof}

Let $\Shv_\mathrm{fpqc}^\mathrm{ad}\subseteq \Fun(\CAlg_\mathrm{ad}, \mS)$ denote the corresponding $\i$-topos. Due to both functors $\delta$ and $U$ respecting the fpqc topology, the presheaf-level adjuncton restricts to an adjunction
$$
U^*\simeq \delta_!:\Shv_\mathrm{fpqc}^\mathrm{nc} \rightleftarrows \Shv_\mathrm{fpqc}^\mathrm{ad}:\delta^*
$$
in which the left adjoint is fully faithful and preserves all limits. This is therefore in particular a geometric morphism of $\i$-topoi.

\begin{remark}\label{Spad def} 
By imitating the proof of \cite[Theorem D.6.3.5]{SAG}, the fpqc Grothendieck topology on $(\CAlg_\mathrm{ad})^\mathrm{op}$ is subcanonical, meaning that adic spectra satisfy flat descent. Indeed, the functor $\mathrm{Spad}:(\CAlg_\mathrm{ad})^\mathrm{op}\to \Shv_\mathrm{fpqc}^\mathrm{ad}$ is the right adjoint of the evident global functions functor. The relationship between the adic spectrum of an adic $\E$-ring and the usual spectrum of an $\E$-ring is encoded in the canonical equivalences $\delta^*\circ \mathrm{Spad}\simeq \Spec$ and $U^*\circ \Spec\simeq\mathrm{Spad}\circ \delta$.
\end{remark}

\subsection{De Rham space and adic pre-completion}

The approach of Remark \ref{ff functors on adic}, where the adic topology is simply disregarded, gives a canonical fully faithful embedding of presheaf $\i$-categories $\Fun(\CAlg, \mS)\to\Fun(\CAlg_\mathrm{ad}, \mS)$. At least when restricted to ordinary adic rings, we will make use of another functor in the same direction, given (in a variant of a notion intoduced by Carlos Simpson in \cite{Simpson}) as follows:

\begin{definition}\label{def of dR}
The \textit{de Rham space} of a functor $X : \CAlg^\heart\to \mS$ (e.g. an ordinary stack) is the functor $X_\mathrm{dR}:\CAlg^\heart_\mathrm{ad}\to \mS$ given by
$$
X_\mathrm{dR}(R) := \varinjlim_{I\subseteq R} X(R/I),
$$
with the colimit ranging over all ideals of definition $I\subseteq R$.
\end{definition}

\begin{remark}
Suppose that $X:\CAlg^\heart\to\mS$ is locally of finite presentation, which is to say that it commutes with filtered colimits. Then $X_\mathrm{dR}(R)\simeq X(R/\sqrt{I})$, which is well-defined since all ideals of definition have the same radical (equivalently: determine the same reduced closed subscheme). In particular, for any commutative ring $R$ equipped with the discrete adic topology, the ideals of definition are precisely the finitely generated nilpotent ideals in $R$. Hence the value of the de Rham space on $\CAlg^\heart\subseteq\CAlg^\heart_\mathrm{ad}$ is given by the classical formula
$X_\mathrm{dR}(R)  \simeq X(R_\mathrm{red}).$
\end{remark}

\begin{remark}\label{de Rham limitcolimit}
The colimit appearing in Definition \ref{def of dR} is filtered. Since limits and colimits in presheaf $\i$-categories are computed object-wise, it follows that the de Rham space construction $X\mapsto X_\mathrm{dR}$ commutes with colimits and finite limits.
\end{remark}

\begin{remark}
For any $X:\CAlg^\heart\to\mS$, there is a canonical map $X\to X_\mathrm{dR}$, where we view $X$ as a functor $\CAlg^\heart_\mathrm{ad}\to\CAlg^\heart\xrightarrow{X}\mS$ as  discussed in Remark \ref{ff functors on adic}. It is obtained by passing to the colimit over ideals of definition $I\subseteq R$ from the maps $X(R)\to X(R/I)$.
\end{remark}

We may use de Rham spaces and the compression map $X\to X^\heart\circ \pi_0$ to define a suitable notion of formal completion in our context.

\begin{definition}\label{completion def}
Let $X$ be a non-connective spectral stack, and $K\to X^\heart$ be a map of ordinary stacks. Then the \textit{adic pre-completion of $X$ along $K$} is defined to be the functor $X_K :\CAlg_\mathrm{ad}\to\mS$ given by the fiber product
$$
\begin{tikzcd}
X_K\ar{r} \ar{d} & K_\mathrm{dR}\circ \pi_0\ar{d}\\
X\ar{r} & X^\heart_\mathrm{dR}\circ\pi_0
\end{tikzcd}
$$
in the $\i$-category $\Fun(\CAlg_\mathrm{ad}, \mS)$. More explicitly, the value of the adic pre-completion of $X$ along $K$ at an adic $\E$-ring $A$ is
$$
X_K(A) := \varinjlim_{I\subseteq\pi_0(A)} X(A)\times_{X^\heart(\pi_0(A)/I)} K(\pi_0(A)/I),
$$
with the colimit ranging over all  the ideals of definition $I\subseteq\pi_0(A)$.
\end{definition}

\begin{remark}\label{formal completion decompresses}
In the setting of Definition \ref{completion def}, since the canonical map $X\to X^\heart_\mathrm{dR}$ factors through the compression map $X\to X^\heart\circ\pi_0,$ we have
\begin{eqnarray*}
X_K &\simeq&  X\times_{(X^\heart\circ\pi_0)} (X^\heart\circ\pi_0) \times_{(X^\heart_\mathrm{dR}\circ \pi_0)} (K_\mathrm{dR}\circ \pi_0)
 \\
&\simeq&
X\times_{(X^\heart\circ\pi_0)} ((X^\heart)_K\circ \pi_0).
\end{eqnarray*}
Hence adic pre-completion always happens on the level of ordinary stacks, and is then merely decompressed into the spectral world.
\end{remark}

\begin{remark}\label{formula for formal completions}
Suppose that the map of ordinary stack $K\to X^\heart$ is locally of finite presentation (or more precisely, of finite presentation to order $1$ in the sense of \cite[Definition  17.4.1.1]{SAG}). Then the diagram
$$
\begin{tikzcd}
\varinjlim_{I} K(\pi_0(A)/I)\ar{r}\ar{d} & K(\pi_0(A)/\sqrt{I})\ar{d}\\
\varinjlim_{I} X^\heart(\pi_0(A)/I)\ar{r} & X^\heart(\pi_0(A)/\sqrt{I})
\end{tikzcd}
$$
with the colimit ranging over all  the ideals of definition $I\subseteq\pi_0(A)$,
is a homotopy pullback, and so adic pre-completion is computed as
$$
X_K(A)\simeq X(A)\times_{X^\heart(\pi_0(A)/\sqrt{I})} K(\pi_0(A)/\sqrt{I}).
$$
\end{remark}

\begin{exun}\label{first exun}
Let $X$ be a non-connective spectral scheme, and $K\subseteq X^\heart$ a closed subscheme, cut out by a locally finitely generated quasi-coherent sheaf of ideals $\mathscr I\subseteq\sO_X$. The last condition is equivalent to the map $K\to X^\heart$ being finitely presented, hence we the formula of Remark \ref{formula for formal completions} applies. For any adic $\E$-ring $A$, the space $X_K(A)$ is therefore the union  of all the connected components in $X(A)\simeq \Map_{\mathrm{SpSch^{nc}}}(\Spec(A), X)$, for which the reduced subscheme $V(\sqrt{I})\subseteq\Spec (\pi_0(A))$ factors under the underlying map of ordinary schemes $\Spec(\pi_0(A))\to X$ through the inclusion $K\to X^\heart$. Since  $V(\sqrt{I})$ is a reduced scheme, this is equivalent by \cite[\href{https://stacks.math.columbia.edu/tag/0356}{Lemma 0356}]{stacks-project}
to the map of topological spaces $|\Spec (\pi_0(A))|\to |X|$ mapping the closed subset $|V(\sqrt I)|$ to $|K|$.
\end{exun}

Under some assumptions,  (sufficient for our purposes,) a ``covering'' of $X$ by affines $\Spec(A)$ induces a corresponding ``covering'' of its adic pre-completion
$X_K$ by adic affines $\mathrm{Spad}(A)$. This is the content of Proposition \ref{informal charts}, but we require a few simple auxiliary results first.

\begin{lemma}\label{Spf is informal comp}
Let $A$ be an adic $\E$-ring with an ideal of definition $I\subseteq A$. There is a canonical equivalence $\Spec(A)_{V(I)}\simeq \mathrm{Spad} (A)$ in the $\i$-category $\Fun(\CAlg_\mathrm{ad}, \mS)$.
\end{lemma}

\begin{proof}
For any adic $\E$-ring $B$, Example \ref{first exun} shows there is a canonical homotopy equivalence
$
\Spec(A)_{V(I)}(B)\simeq \Map_{\CAlg_\mathrm{ad}}(A, B),
$
since a map of adic $\E$-rings is precisely a map of $\E$-rings which point-set-topologically (but not necessarily scheme-theoretically) respects the closed subschemes defined by ideals of definition. That is of course precisely the definition of the adic spectrum $\mathrm{Spad}(A)$.
\end{proof}

\begin{lemma}\label{base-change of informal comp}
Let $X'\to X$ be a map of non-connective spectral stacks, and denote $K':=(X')^\heart\times_{X^\heart} K$. The canonical map $X'\times_X X_K\to(X')_{K'}$ is an equivalence.
\end{lemma}

\begin{proof}
It suffices by Remark \ref{formal completion decompresses} to assume that all the stacks in sight are ordinary. Then we have
\begin{eqnarray*}
X'\times_X X_K &\simeq & X'\times_X X\times_{(X_\mathrm{dR}\circ \pi_0)} (K_\mathrm{dR}\circ \pi_0)\\
&\simeq & X'\times_{(X_\mathrm{dR}\circ \pi_0)}(K_{\mathrm{dR}}\circ \pi_0)\\
&\simeq & X'\times_{(X'_\mathrm{dR}\circ \pi_0)} (X'_\mathrm{dR}\circ \pi_0)\times_{(X_\mathrm{dR}\circ \pi_0)} (K_\mathrm{dR}\circ \pi_0)\\
&\simeq & X'\times_{(X'_\mathrm{dR}\circ \pi_0)}((X'\times_X K)_\mathrm{dR}\circ \pi_0)\\
&\simeq & (X')_{K'},
\end{eqnarray*}
where the only equivalence to require justification is the second-to-last. There we used the fact that the de Rham space functor $(-)_\mathrm{dR}:\Fun(\CAlg^\heart, \mS)\to \Fun(\CAlg_\mathrm{ad}^\heart, \mS)$ commutes with products, which is clear from Definition \ref{def of dR}.
\end{proof}

The statement of the following proposition involves a number of colimits, formed in different $\i$-categories. For the sake of clarity regarding where colimits are being formed, we will therefore denote the colimit of a diagram $C_i$ in an $\i$-category $\mC$ by $\varinjlim^{\mC}_i C_i$.

\begin{prop}\label{informal charts}
Let $X$ be a non-connective spectral stack, and let $K\to X^\heart$ be a closed immersion of finite presentation of ordinary stacks. 
Suppose that  $X\simeq \varinjlim_i^{\Shv_\mathrm{fpqc}^\mathrm{nc}} \Spec(A_i)$ for some diagram of $\E$-rings $A_i$, and suppose that that the closed subscheme $K\times_{X^\heart} \Spec(\pi_0(A_i))\subseteq\Spec (\pi_0(A_i))$ is defined by a finitely generated ideal $I_i\subseteq \pi_0(A_i)$ for evey $i$. Then the canonical map
$
\varinjlim_i^{\Fun(\CAlg_\mathrm{ad}, \mS)} \mathrm{Spad}(A_i)\to
X_K
$
 in $\Fun(\CAlg_\mathrm{ad}, \mS)$ induces an isomorphism upon sheafification. More precisely, it exhibits $\varinjlim_i^{\Shv_\mathrm{fpqc}^\mathrm{ad}} \mathrm{Spad}(A_i)$  as the fpqc sheafification of the adic pre-completion  $X_K$.
\end{prop} 
\begin{proof}
Let $L^\mathrm{ad} : \Fun(\CAlg_\mathrm{ad}, \mS)\to \Shv_\mathrm{fpqc}^\mathrm{ad}$ denote the fpqc sheafification functor. From the assumed colimit presentation of $X$, and the fact that the canonical fully faithful inclusion $\delta_!:\Shv_\mathrm{fpqc}^\mathrm{nc}\subseteq\Shv_\mathrm{fpqc}^\mathrm{ad}$ is a left adjoint and therefore preserves colimits, we find that the equivalence $X\simeq \varinjlim^{\Shv_\mathrm{fpqc}^\mathrm{ad}}\Spec(A_i)$ holds in $\Shv_\mathrm{fpqc}^\mathrm{ad}$. Denoting by $L^\mathrm{ad}$ the fpqc sheafification, the universality of colimits in the $\i$-topos $\Shv_\mathrm{fpqc}^\mathrm{ad}$ implies that
$$
L^{\mathrm{ad}}(X_K)
\simeq \varinjlim_i{}^{\Shv_\mathrm{fpqc}^\mathrm{ad}} \Spec(A_i)\times_X L^\mathrm{ad}(X_K).
$$
Since both $X$ and $\Spec(A_i)$ already satisfy flat descent, and because $L^\mathrm{ad}$, just like any sheafification, preserves finite limits, we have
$$
\Spec(A_i)\times_X L^\mathrm{ad}(X_K)\simeq L^\mathrm{ad}(\Spec(A_i)\times_X X_K).
$$
Lemma \ref{base-change of informal comp} garners a canonical equivalence
$$
\Spec(A_i)\times_X X_K\simeq (\Spec(A_i)\times_X X)_{\Spec(A_i)\times_X K}\simeq \Spec(A_i)_{V(I_i)},
$$
in the $\i$-category $\Fun(\CAlg_\mathrm{ad}, \mS)$, and the right-most turn is identified with $\mathrm{Spad}(A_i)$ by Lemma \ref{Spf is informal comp}. As noted in Remark \ref{Spad def}, the functor $\mathrm{Spad}(A_i):\CAlg_\mathrm{fpqc}\to\mS$ satisfies flat descent, meaning it is invariant under sheafification.
The desired  result follows.
\end{proof}

\subsection{Quasi-coherent sheaves in the adic world}
We purse the idea that some amount of `functor of points' algebraic geometry carries over to the setting of $\Fun(\CAlg_\mathrm{ad}, \mS)$ a bit further, by considering a notion of quasi-coherent sheaves. As usual, we first discuss what happens in the affine setting:

\begin{definition}\label{complete def}
Let $A$ be an adic $\E$-ring, with an ideal of definition $I\subseteq\pi_0(A)$.
Recall from \cite[Definition 7.3.1.1]{SAG} that \textit{the $\i$-category of $I$-complete $A$-modules} is defined to be the right orthogonal $\Mod_A^{\mathrm{Cpl}(I)} :=\big(\Mod_A^{\mathrm{Loc}(I)}\big)^\perp$ to the full subcategory of $I$-local $A$-modules $\Mod_A^{\mathrm{Loc}(I)}\subseteq\Mod_A$. The  inclusion $\Mod_A^{\mathrm{Cpl}(I)}\subseteq\Mod_A$ admits by \cite[Notation 7.3.1.5]{SAG} a left adjoint $M\mapsto M^\wedge_I$, called \textit{$I$-completion}.
\end{definition}

\begin{remark}\label{Greenlees-May duality}
The pair $(\Mod_A^{\mathrm{Loc}(I)}, \Mod_A^{\mathrm{Cpl}(I)})$ comprises according to \cite[Proposition 7.3.1.4]{SAG} another semi-orthogonal decomposition of  the stable $\i$-category $\Mod_A$. In particular, every $A$-module $M$ fits into a cofiber sequence
$$
L_I(M)\to M\to M^\wedge_I.
$$
This semi-orthogonal decomposition is related to the one we discussed in Remark \ref{semi-orth loc-nilp}; by \cite[Proposition 7.3.1.7]{SAG}, the adjunction
$$
(-)^\wedge_I: \Mod_A^{\mathrm{Nil}(I)}\rightleftarrows \Mod_A^{\mathrm{Cpl}(I)}:\Gamma_I
$$
is an adjoint equivalence of $\i$-categories. This is a form of Greenlees-May duality \cite{Greenlees-May}.
\end{remark}

\begin{cons}\label{formal pushpull}
Let $f:A\to B$ be a map of adic $\E$-rings, and let $I\subseteq\pi_0(A)$ and $J\subseteq\pi_0(B)$ be ideals of definition. As an $\E$-ring map, $f$ induces an adjunction between $\i$-categories of modules $B\o_A-: \Mod_A\rightleftarrows \Mod_B$. It follows by \cite[Corollary 7.3.3.6]{SAG} from the continuity of $\pi_0(f)$ that the forgetful functor, e.g.~right adjoint of this adjunction, sends $J$-complete $B$-modules to $I$-complete $A$-modules. Thus it restricts to a functor $f_* :\Mod_B^{\mathrm{Cpl}(J)}\to\Mod_A^{\mathrm{Cpl}(I)}$. It follows by abstract nonsense that this functor admits a left adjoint $f^*$, given by the composite
$$
\Mod_A^{\mathrm{Cpl}(I)}\subseteq\Mod_A\xrightarrow{B\o_A -}\Mod_B\xrightarrow{(-)^\wedge_J}\Mod_B^{\mathrm{Cpl}(I)}.
$$
In summary, we obtain an adjunction
$f^* : \Mod_A^{\mathrm{Cpl}(I)}\rightleftarrows \Mod_B^{\mathrm{Cpl}(J)}: f_*$ between two presentable stable $\i$-categories.
\end{cons}

 \begin{ex}\label{examples of Modcpl}
 The bivariant functoriality  $f^*$ and $f_*$ of complete modules on adic $\E$-rings is paritcularly interesting when one or both of the adic topologies is discrete:
 \begin{enumerate}[label = (\alph*)]
\item  Let $A$ be an adic $\E$-ring with an ideal of definition $I\subseteq\pi_0(A)$. Consider the canonical adic $\E$-ring map $f :\delta(A)\to A$, where the domain is equiped with the discrete adic topology.
 Since $\Mod_{\delta(A)}^{\mathrm{Cpl}(0)}\simeq \Mod_A,$ the adjunction $f^*\dashv f_*$ of Construction \ref{formal pushpull} recovers the adjunction $(-)^\wedge_I : \Mod_A^{\mathrm{Cpl}(I)}\rightleftarrows \Mod_A$.
\item Let $f:A\to B$ be a map in $\CAlg\subseteq\CAlg_\mathrm{ad}$.
Then the adjunction $f^*\dashv f_*$ of Construction \ref{formal pushpull} is the usual tensor-forgetful adjunction $B\o_A -:\Mod_A\rightleftarrows \Mod_B.$
\end{enumerate}
 \end{ex}

 The $\i$-category $\Mod_A^{\mathrm{Cpl}(I)}$ also has by \cite[Variant 7.3.5.6]{SAG} a canonical symmetric monoidal structure $\widehat{\o}_A$, given by $M\widehat{\o}_A N :=(M\o_A N)^\wedge_I.$ It is clear from its definition in Construction \ref{formal pushpull} that the functor $f^*$ is symmetric monoidal with respect to this symmetric monoidal structure. In fact, the full functoriality of the construction $A\mapsto \Mod_A^{\mathrm{Cpl}(I)}$ may be encoded as a functor $\CAlg_\mathrm{ad}\to\CAlg(\Pr^\mathrm L)$.

 \begin{definition}\label{formal QCoh}
The functor of \textit{quasi-coherent sheaves}
$$\QCoh :\Fun(\CAlg_\mathrm{ad}, \mS)^\mathrm{op}\to \CAlg(\PrL)$$
is defined by right Kan extension from the above-described functor $A\mapsto \Mod_A^{\mathrm{Cpl}(I)}$. That is to say, the $\i$-category of quasi-coherent sheaves  is given by
$$
\QCoh\big(\varinjlim_i \mathrm{Spad}(A_i)\big)\simeq \varprojlim_i \Mod^{\mathrm{Cpl}(I_i)}_{A_i}.
$$
 \end{definition}

 \begin{remark}
 Let us unpack the content of Definition \ref{formal QCoh}.  For any morphism $f : X\to Y$ in $\Fun(\CAlg_\mathrm{ad}, \mS)$, we obtain from the definition of quasi-coherent sheaves an adjunction
$$
f^*:\QCoh(Y)\rightleftarrows \QCoh(X) : f_*,
$$
whose left adjoint $f^*$ we call \textit{pullback along $f$}, and whose right adjoint $f_*$ we call \textit{pushforward along $f$}. The $\i$-categories $\QCoh(X)$ and $\QCoh(Y)$ are symmetric monoidal with respective operations $\widehat{\o}_{\sO_X}$ and $\widehat{\o}_{\sO_Y}$, with respect to which the pullback functor $f^*$ is symmetric monoidal.
\end{remark}

\begin{remark}
As observed in Example \ref{examples of Modcpl}, the functor $\Mod^{\mathrm{Cpl}}:\CAlg^\mathrm{cpl}_\mathrm{ad}\to \CAlg(\Pr^{\mathrm L})$  recovers, upon restriction to the full subcategory of discrete adic $\E$-rings $\CAlg\subseteq\CAlg_\mathrm{ad}$, the usual bivariant functoriality of $\Mod : \CAlg\to \CAlg(\Pr^\mathrm L).$ Consequently, the functor $\QCoh$ of Definition \ref{formal QCoh} restricts on the full subcategory $\Shv_\mathrm{fpqc}^\mathrm{nc}\subseteq\Fun(\CAlg, \mS)\subseteq\Fun(\CAlg_\mathrm{ad}, \mS)$ to the eponymous functor $\QCoh$ of  \cite[Defintion 1.4.2]{ChromaticCartoon}.
\end{remark}

\begin{prop}\label{QCoh don't care about sheafification}
The functor $\QCoh$ of Definition \ref{formal QCoh} is invariant under fpqc sheafification.
\end{prop}

\begin{proof}
Just like in the non-formal analogue of this result \cite[Proposition  6.2.3.1]{SAG}, this boils down to proving that the construction $A\mapsto \Mod_A^{\mathrm{Cpl}(I)}$ satisfies flat descent. That is, let $A\to B$ be a faithfully flat map of adic $\E$-rings, with ideals of definition $I\subseteq\pi_0(A)$ and $J\subseteq\pi_0(B)$. Without loss of generality, we may assume that $f(I)\subseteq J$ (else we replace $I$ by a sufficiently high power of it, which leaves the adic topology unchanged). In light of the Beck-Chevalley comonadic formulation of descent \cite[Section 4.7.5]{HA}, we must show that the adjunction
$$
(B\o_A -)^\wedge_J:\Mod_A^{\mathrm{Cpl}(I)}\rightleftarrows \Mod_B^{\mathrm{Cpl}(J)}
$$
of Construction \ref{formal pushpull} is comonadic. By the Barr-Beck-Lurie Comonadicity Theorem, it suffices to show that the left adjoint $M\mapsto (B\o_A M)^\wedge_J$ is conservative, and preserves all split totalizations.

In fact, we show it preserves all limits. The functor $B\otimes_A -:\Mod_A\to\Mod_B$ is comonadic, and as such preserves limits, on the account $A\mapsto\Mod_A$ satisfying flat descent \cite[Corollary  D.6.3.3]{SAG}. The inclusion $\Mod_A^{\mathrm{Cpl}(I)}\to\Mod_A$ is a left adjoint and hence also preserves limits. Finally we must show that the completion functor $(-)^\wedge_J  :\Mod_B\to \Mod_B^{\mathrm{Cpl}(J)}$ preserves limits. Since we already know that $\Mod_B\to\Mod_B^{\mathrm{Cpl}(J)}$ preserves limits, it suffices to show that $N\mapsto N^\wedge_J$  preserves limits as an endofunctor of the $\i$-category $\Mod_B$. To make that evident, we write it as a composite of limit-preserving functors
$$
N^\wedge_J\simeq \mathrm{cofib}(L_J(N)\to N)\simeq \Sigma(\mathrm{fib}(L_J(N)\to N));
$$
here suspension commutes with limits thanks to the stability of $\Mod_B$.

It remains to show that $M\mapsto (B\o_A M)^\wedge_J$ is conservative. Thanks to stability of the $\i$-categories in sight, that amounts to verifying that  an equivalence $(B\o_A M)^\wedge_J\simeq 0$ for an $I$-complete $A$-module $M$ implies that $M\simeq 0$. So suppose that $(B\o_A M)^\wedge_J\simeq 0$. Using the fiber sequence
$$
L_J(B\o_A M)\to B\o_A M\to (B\o_A M)^\wedge_J,
$$
in the $\i$-category $\Mod_B$, we conclude that the $B$-module $B\o_A M$ is $J$-local. In particular, since the ideal of $B$ generated by the image $f(I)$ is  by assumption contained in the ideal of definition $J$, this implies by \cite[Remark 7.2.4.7]{SAG} that $B\o_A M$  is $I$-local as an $A$-module. Therefore, we find using the compatibility of smashing with localization along finitely generated ideals \cite[Corollary 7.2.4.12]{SAG}, that
$$
B\o_A M\simeq L_I(B\o_A M)\simeq B\o_A L_I(M).
$$
This equivalence is induced from the canonical map $L_I(M)\to M$ by applying the functor $B\o_A-$. But since the latter is, again by  \cite[Corollary  D.6.3.3]{SAG}, comonadic, it is in particular conservative. Hence $L_I(M)\simeq M$, showing that $M$ is an $I$-local $A$-module. It is also $I$-complete by assumption, which is only possible if $M\simeq 0$.
\end{proof}

\subsection{Complete adic $\E$-rings and formal completion}
The preceding  discussion of quasi-coherent sheaves in the previous subsection concerned complete modules, yet we have so far not discussed any corresponding completeness assumption on adic $\E$-rings. We do so now, which leads us to a setup closer to what might rightfully be called formal (non-connective spectral) algebraic geometry, as opposed to the adic version that we have been discussing up to this point.

\begin{definition}
An adic $\E$-ring $A$ is \textit{complete} if it is $I$-complete, in the sense of Definition \ref{complete def}, for some (and therefore any) ideal of definition $I\subseteq\pi_0(A)$. Let $\CAlg_\mathrm{ad}^\mathrm{cpl}\subseteq\CAlg_\mathrm{ad}$ denote the full subcategory spanned by complete adic $\E$-rings.
\end{definition}

The functor of $I$-adic completion $A\mapsto A^\wedge_I$,  in the sense of \cite[Notation 7.3.1.5]{SAG},  provides a left adjoint to the subcategory inclusion $V:\CAlg_\mathrm{ad}^\mathrm{cpl}\subseteq\CAlg_\mathrm{ad}$. The adjunction $(-)^\wedge_-:\CAlg_\mathrm{ad}\rightleftarrows \CAlg_\mathrm{ad}^\mathrm{cpl}:V$ induces  adjunctions between copresheaf $\i$-categories
$$
\Fun(\CAlg_\mathrm{ad}, \mS)\overset{((-)^\wedge_-)_!}{\underset{((-)^\wedge_-)^*} \rightleftarrows}\Fun(\CAlg_\mathrm{ad}^\mathrm{cpl}, \mS) \overset{V_!}{\underset{V^*}\rightleftarrows}\Fun(\CAlg_\mathrm{ad}, \mS)
$$
through left Kan extension and precomposition. These two adjunctions are connected as
$$
V_!\dashv V^*\simeq ((-)^\wedge_-)_!\dashv ((-)^\wedge_-)^*,
$$
where the middle equivalence stems from the adjunction $(-)^\wedge_- \dashv V$ by abstract nonsense.

\begin{prop}
The functor $V_! : \Fun(\CAlg^{\mathrm{cpl}}_\mathrm{ad}, \mS)\to\Fun(\CAlg_\mathrm{ad}, \mS)$ is fully faithful.
\end{prop}

\begin{proof}
For any $I$-complete  $\E$-ring $A$, the canonical map $A\to A^\wedge_I$ is an equivalence.  As consequence, we find that the unit map $\id\to V^*\circ V_!\simeq ((-)^\wedge_-)_!\circ V_!\simeq ((-)^\wedge_-\circ V)_!$ of the adjunction $V_!\dashv V^*$ is an equivalence. But that is precisely equivalent to the left adjoint being fully faithful. 
\end{proof}

We view the functor $V_!$ as an implicit inclusion of functors $\CAlg_\mathrm{ad}^\mathrm{cpl}\to\mS$ into functors $\CAlg_\mathrm{ad}\to \mS$, and will as such often omit it from notation.

\begin{definition}\label{real def of formal completion}
Given a functor $X:\CAlg_\mathrm{ad}\to\mS$, we call the functor $X^\wedge:\CAlg_\mathrm{ad}^\mathrm{cpl}\to\mS$, given by restriction $X^\wedge :=X\vert_{\CAlg^\mathrm{cpl}_\mathrm{ad}},$ the \textit{completion of $X$}.
\end{definition}

\begin{exun}
For an adic $\E$-ring $A$, its \textit{formal spectrum} is defined to be the completion $\Spf(A) :=\mathrm{Spad}(A)^\wedge$ of its adic spectrum. It is clear that any map of adic $\E$-rings $A\to B$, which induces an equivalence on completions, also induces an equivalence $\Spf(A)\simeq \Spf(B)$. In particular, we have $\Spf(A)\simeq \Spf(A^\wedge_I)$ for all adic $\E$-rings $A$. But in $\Fun(\CAlg_\mathrm{ad}^\mathrm{cpl}, \mS)$, the latter functor is corepresentable by $A^\wedge_I$. The canonical fully faithful inclusion $V_!:\Fun(\CAlg_\mathrm{ad}^\mathrm{cpl}, \mS)\subseteq\Fun(\CAlg_\mathrm{ad}, \mS)$ therefore identifies the formal spectrum $\mathrm{Spf}(A)$ of an adic $\E$-ring $A$ with the adic spectrum  $\mathrm{Spad}(A^\wedge_I)$ of its completion $A^\wedge_I$.
\end{exun}

\begin{remark}
Rephrasing the definition, the  completion functor $(-)^\wedge :\Fun(\CAlg_\mathrm{ad}, \mS)\to\Fun(\CAlg_\mathrm{ad}^\mathrm{cpl}, \mS)$ is the functor $V^*$ discussed above. Since we have already seen that $V^*\simeq ((-)^\wedge_-)_!$, this means that completion admits an alternative description as a left Kan extension. Indeed, if we write $X\simeq \varinjlim_i \mathrm{Spad}(A_i)$ in $\Fun(\CAlg_\mathrm{ad}, \mS)$ (as we always can in a presheaf $\i$-category), then its completion admits colimit descriptions
$$X^\wedge\simeq \varinjlim_i \Spf(A_i)\simeq \varinjlim_i \mathrm{Spad}((A_i)^\wedge_{I_i}).$$
\end{remark}

\begin{remark}
Though we will not need such generality, the conclusion of Example \ref{first exun} holds when $X$ is a non-connective spectral Deligne-Mumford stack. Under the added assumption that $X$ is connective, this shows that our definition of formal completion in this case coincides with that of \cite[Definition 8.1.6.1]{SAG}.
\end{remark}

We finally come to our main object of interest in this foray into formal geometry: the formal completion. We define it by the just-introduced completion process from the adic pre-completion of Definition \ref{completion def}.

\begin{definition}\label{def fc}
Let $X$ be a non-connective spectral stack, and $K\to X^\heart$ be a map of ordinary stacks. The \textit{formal completion of $X$ along $K$} is defined to be  the completion $X^\wedge_K$ of the adic pre-completion $X_K$ of $X$ along $K$.
\end{definition}

\begin{remark}\label{fc explicit}
Unwinding the web of definitions, the value of the formal completion $X^\wedge_K : \CAlg_\mathrm{ad}^\mathrm{cpl}\to\mS$ at a complete adic $\E$-ring $A$ is given as the filtered colimit
$$
X^\wedge_K(A) := \varinjlim_{I\subseteq\pi_0(A)} X(A)\times_{X^\heart(\pi_0(A)/I)} K(\pi_0(A)/I),
$$
ranging over all  the ideals of definition $I\subseteq\pi_0(A)$. Unlike the adic pre-completion, in the case of which the above formula remains valid for all adic $\E$-rings, it is only valid for complete adic $\E$-rings in the case of the formal completion. As a functor $X^\wedge_K :\CAlg_\mathrm{ad}\to\mS$, the values on arbitrary adic $\E$-rings are harder to describe.
\end{remark}

\begin{remark}\label{fc is decompress}
Adic pre-completion is defined in Definition \ref{completion def} through decompression. An ordinary stack $X:\CAlg\to\mS$ appears in the definition of $X_K\in\Fun(\CAlg_\mathrm{ad}, \mS)$ left Kan extension along the inclusion of discretely topologized adic $\E$-rings $\CAlg\to\CAlg_\mathrm{ad}$, and said inclusion factors through the inclusion $\CAlg_\mathrm{ad}^\mathrm{cpl}\subseteq\CAlg_\mathrm{ad}$, it follows that $X^\wedge\simeq X$. Consequently there is a canonical equivalence
$$
X^\wedge_K\simeq X\times_{(X^\heart\circ \pi_0)} \big((X^\heart)^\wedge_{K^\heart}\circ\pi_0\big)
$$
in the $\i$-category $\Fun(\CAlg_\mathrm{ad}^\mathrm{cpl}, \mS)$, showing that formal completion is also defined via decompression.
\end{remark}

\begin{remark}
Our notion of formal completion (as well as adic pre-completion) is closely related to the relative de Rham space of \cite[Definition 18.2.1.1]{SAG}. Indeed, when restricted to the full subcategory $\CAlg\subseteq\CAlg_\mathrm{ad}^\mathrm{cpl}$ of $\E$-rings with the (always complete) discrete adic topology, and for $X$ a connective spectral stack, we have
$$
X^\wedge_K\vert_{\CAlg}\simeq (K/X)_\mathrm{dR}.
$$
However, we prefer to use the completion notation, in part because we find it  more geometrically evocative. But perhaps more importantly, the relative de Rham space notation $(K/X)_\mathrm{dR}$ suggests the existence of a map $K\to X$. Such a map indeed always exists in the connective case, where it may be obtained by composing the map $K\to X^\heart$ with the canonical map $X^\heart\to X$. But in the non-connective setting, $K$ and $X$ are no longer ``on the same footing'', which we believe is better reflected in the notation $X^\wedge_K$.
\end{remark}

Now that we have defined formal completion, our next step is to relate quasi-coherent sheaves on it to quasi-coherent sheaves with specified support that we had discussed a little while back. After some preliminary discussion in Remark \ref{completion as pullback}, that is encapsulated in Proposition \ref{all about formal QCoh}.

\begin{remark}\label{completion as pullback}
The counit $V_!\circ V^*\to \mathrm{id}$ of the adjunction $V_!\dashv V^*$ gives rise to a canonical map $X^\wedge\to X$ in the $\i$-category $\Fun(\CAlg_\mathrm{ad}, \mS)$ for any functor $X:\CAlg_\mathrm{ad}\to\mS$. In the special case when $X$ is a non-connective spectral stack and $K\subseteq X^\heart$, this may be combined with the canonical map $X_K\to X$ on the level of adic pre-completion, to obtain a canonical map $f:X^\wedge_K \to X$. This in turn induces a pullback-pushforward adjunction on quasi-coherent sheaves
\begin{equation}\label{adjunction formal pushpull}
f^* : \QCoh(X)\rightleftarrows \QCoh(X^\wedge_K): f_*.
\end{equation}
A special case of this was
discussed in Example \ref{examples of Modcpl}, namely the case that $X =\Spec(A)$ and $K=V(I)$ for a finitely generated radical ideal $I\subseteq\pi_0(A)$. In analogy to that special case, we in the general case still view $f_*$ as a forgetful functor and as such often suppress it from notation; somewhat justifiably, since it is fully faithful. Consequently, the pullback  $f^*$ is to be viewed as completion along $K$, and as such denoted $f^*(\sF) \simeq \sF^\wedge_K$. 
\end{remark}

\begin{prop}\label{all about formal QCoh}
Let $U\to X$ be a open immersion of geometric non-connective spectral stacks. Suppose that $X$ and the reduced closed substack $K :=X^\heart-U^\heart$ of finite presentation. Then:
\begin{enumerate}[label = (\alph*)]
\item The canonical functor $\QCoh(X^\wedge_K)\to \QCoh(X)$ is fully faithful. \label{jedan}
\item The fully faithful  embedding from \ref{jedan} exhibits a semi-orthogonal decomposition $(\QCoh(U), \QCoh(X^\wedge_K))$ of the stable $\i$-category $\QCoh(X)$. \label{dva}
\item The adjunction\label{tri}
$$
(-)^\wedge_K : \QCoh_K(X)\simeq\QCoh(X^\wedge_K) : \Gamma_K
$$
is a symmetric monoidal adjoint equivalence of $\i$-categories.
\end{enumerate}
\end{prop}

The proof of Proposition \ref{all about formal QCoh} requires some preliminaries regarding the behavior of quasi-coherent sheaves in the setting of functors on complete adic $\E$-rings.

\begin{lemma}\label{QCoh don't care about completion}
For any $X\in\Fun(\CAlg_\mathrm{ad}, \mS)$, the adjunction $\QCoh(X)\rightleftarrows\QCoh(X^\wedge)$, induced by the canonical map $X^\wedge\to X$ on quasi-coherent sheaves, is an adjoint equivalence of $\i$-categories.
\end{lemma}

\begin{proof}
Recall that both functors $(-)^\wedge$ and $\QCoh$ are defined as Kan extensions of the functors on adic $\E$-rings $A\mapsto A^\wedge_I$ and $A\mapsto \Mod_A^{\mathrm{Cpl}(I)}$ respectively. It therefore suffices to show that the canonical adic $\E$-ring map $A\to A^\wedge_I$ induces an equivalence upon $I$-complete modules for any adic $\E$-ring $A$ and ideal of definition $I\subseteq\pi_0(A)$. The adjunction it induces on complete modules is 
$$
A^\wedge_I\widehat{\o}_A-:\Mod_A^{\mathrm{Cpl}(I)}\rightleftarrows \Mod_{A^\wedge_I}^{\mathrm{Cpl}(I)},
$$
and  its unit and counit are equivalences  by an application of \cite[Proposition 7.3.5.1]{SAG}.
\end{proof}

\begin{lemma}\label{formal charts}
With $X\simeq \varinjlim_i\Spec(A_i)$ and $K\subseteq X^\heart$ as in the setting of Proposition \ref{informal charts},  the canonical  map $\varinjlim_i\Spf(A_i)\to X^\wedge_K$ induces upon quasi-coherent sheaves an equivalence of $\i$-categories $\QCoh(X^\wedge_K)\simeq \varprojlim_i\Mod_{A_i}^{\mathrm{Cpl}(I_i)}.$
\end{lemma}

\begin{proof}
We already know from Proposition \ref{informal charts} that the canonical map $f:\varinjlim_i\mathrm{Spad}(A_i)\to X_K$ induces an equivalence upon sheafification. It follows from Proposition \ref{QCoh don't care about sheafification} that $f$ also induces an equivalence on $\QCoh$. The canonical map that we are looking at here $f^\wedge:\varinjlim_i \Spf(A_i)\to X^\wedge_K$ is obtained from $f$ by completion. Since this also leaves quasi-coherent sheaves unchanged by Lemma \ref{QCoh don't care about completion}, the result follows.
\end{proof}

\begin{proof}[Proof of Proposition \ref{all about formal QCoh}]
For \ref{jedan}, we are asserting that the right adjoint of the adjunction \eqref{adjunction formal pushpull} of Remark \ref{completion as pullback} is fully faithful. That is equivalent to the counit $f^*f_*\to \mathrm{id}$ being an equivalence in the $\i$-category $\QCoh(X^\wedge_K)$. It follows from Lemma \ref{formal charts} that the adjunction \eqref{adjunction formal pushpull} is obtained by passage to the limit from the diagram of adjunctions
$$
(-)^\wedge_{I_i}:\Mod_{A_i}\rightleftarrows \Mod_{A_i}^{\mathrm{Cpl}(I_i)}.
$$
The counits of these are clearly equivalences, so the same holds for \eqref{adjunction formal pushpull}. That concludes the proof of \ref{jedan}.

The adjunction in question for \ref{tri} is obtained as the composite of two adjunctions
$$
\QCoh_K(X) \underset{\,\,\Gamma_K}\rightleftarrows\QCoh(X) \overset{(-)^\wedge_K}{\rightleftarrows} \QCoh(X_K^\wedge).
$$
Just like before, it follows that the above adjunction is obtained by passage to the limit from the diagram of adjunctions of $\i$-categories
$$
\Mod_{A_i}^{\mathrm{Nil}(I_i)}\underset{\,\,\Gamma_{I_i}}\rightleftarrows \Mod_{A_i} \overset{(-)^\wedge_{I_i}}\rightleftarrows \Mod_{A_i}^{\mathrm{Cpl}(I_i)}.
$$
As observed in Remark \ref{Greenlees-May duality}, this composite adjunction is an adjoint equivalence by  \cite[Proposition 7.3.1.7]{SAG}. It follows that the adjunction $(-)^\wedge_K \dashv \Gamma_K$ is also an adjoint equivalence, and since its left adjoint is symmetric monoidal by design, it is also an equivalence of symmetric monoidal $\i$-categories.

The proof of \ref{dva} is the same kind of reduction to the (formally) affine case,  and invocation of Remark \ref{Greenlees-May duality}, or more precisely, of \cite[Proposition 7.3.1.4]{SAG}.
\end{proof}

\begin{remark}[Fracture square]
Point \ref{dva} of Proposition \ref{all about formal QCoh} implies that $(\sF\vert_U)^\wedge_K\simeq 0$ for any quasi-coherent sheaf $\sF$ on $X$. On the other hand, the diagram
$$
\begin{tikzcd}
\sF\ar{r}{} \ar{d}{} & \sF^\wedge_K \ar{d}{}\\
\sF\vert_U\ar{r}{} & (\sF^\wedge_K)\vert_U,
\end{tikzcd}
$$
is a pullback square in the $\i$-category $\QCoh(X)$, which may be seen, through a reduction to affines as in the proof of Proposition \ref{all about formal QCoh}, either by appealing to \cite[Remark 7.3.1.6]{SAG}, or by repeating the same argument as is given there to prove it. Indeed, consider the fiber sequence
$$
\sG\to \sF\to \sF\vert_U\times_{(\sF^\wedge_K)\vert_U}\sF^\wedge _K
$$
in $\QCoh(X)$, given by the commutativity of the above diagram. Both the completion and restriction functor $\sF\mapsto \sF^\wedge_K$ and $\sF\mapsto \sF\vert_U$ preserve finite limits. Applying restriction to $U$ to the above fiber sequence
$$
\sG\vert_U\simeq \fib(\sF\vert_U\to \sF\vert_U\times_{(\sF^\wedge_K)\vert_U}(\sF^\wedge_K)\vert_U)\simeq 0.
$$
Because $(\sF\vert_U)^\wedge_K\simeq 0$ is a basic consequences of the semi-direct decomposition of quasi-coherent sheaves into local and complete quasi-coherent sheaves, we also find that
\begin{eqnarray*}
\sG^\wedge_K
&\simeq&
\fib(\sF^\wedge_K\to (\sF\vert_U)^\wedge_K\times_{((\sF^\wedge_K)\vert_U)^\wedge_K}\sF^\wedge _K)\\
&\simeq&
\fib(\sF^\wedge_K\to 0\times_0\sF^\wedge _K)\\
&\simeq& 0. 
\end{eqnarray*}
Consequently $\sG\in \QCoh(U)\cap\QCoh(X_K^\wedge)$ and so $\sG\simeq 0$, showing that the above  commutative square is indeed Cartesian.
\end{remark}

\section{The chromatic filtration on the stack of oriented formal groups}\label{Section 3}

In this sections, we will specialize the discussion of open substacks, quasi-coherent sheaves with fixed support, and formal completions, to the special case of the non-connective spectral stack $\M$ - the primary object of interest of this paper.

\begin{remark}
We will mostly work $(p)$-locally in this section. That means that we will be using the $\E$-ring $(p)$-completion map $S\to S_{(p)}$ (resp.\ the ordinary commutative ring $(p)$-completion map $\mathbf Z\to\mathbf Z_{(p)}$) to base-change certain non-connective spectral stacks (resp.\ ordinary stacks) $X$ as $X\o  S_{(p)}$ (resp.\ $X\o_{\mathbf Z}\mathbf Z_{(p)}$), obtaining in this way stacks defined over the base $\Spec(S_{(p)})$. Thanks to    \cite[Remark 1.1.4.2]{SAG}, the $(p)$-localization map $S\to S_{(p)}$ (and hence $\mathbf Z\to \mathbf Z_{(p)}$ also) is \'etale, so the outlined $(p)$-localization procedure falls into the scope of Remark \ref{remark about etale base-change}. In particular, it is yet another instance of decompression.
\end{remark}

\subsection{The height filtration on $\mathcal M_\mathrm{FG}^\heart$}

Our goal necessitates first a brief review of the reduced closed and open substacks of th ordinary stack of formal groups $\mathcal M_\mathrm{FG}^\heart$. More precisely, we fix throughout a prime $p$, and consider the reduced closed and open substacks of the $(p)$-localization $\mathcal M_\mathrm{FG}^\heart\o_{\mathbf Z}\mathbf Z_{(p)}$. That is precisely the theory of heights of formal groups.

\begin{definition}[{\cite[Definition 4.4.1]{Elliptic 2}}]
Fix a prime $p$ and an integer $n\ge 1$.
A formal group $\w{\G}$ over a commutative ring $R$ \textit{has height $\ge n$} if and only if $p=0$ in $R$, and the  $p$-power map with respect to the group operation $[p] :\w{\G}\to \w{\G}$ factors through $n$ stages of the relative Frobenius tower
$$
\w{\G}\xrightarrow{\varphi_{\w{\G}}} \w{\G}^{(p)}\xrightarrow{\varphi_{\w{\G}^{(p)}}} \w{\G}^{(p^2)}\xrightarrow{\varphi_{\w{\G}^{(p^2)}}} \w{\G}^{(p^3)}\to\cdots.
$$
\end{definition}

\begin{remark}\label{Landideal}
By \cite[Proposition 4.4.10]{Elliptic 2}, the condition that $\w{\G}$ has height $\ge n$ is closed, determining a closed locus of $\Spec(R)$. The ideal $\mathfrak I^{\w{\G}}_n\subseteq R$ that cuts out this locus is called  the \textit{$n$-th Landweber ideal of $\w{\G}$}. It is always finitely generated, and when the formal group $\w{\G}$ is presented by a $p$-typical formal group law, takes on the familiar form
$
\mathfrak I^{\w{\G}}_n = (p, v_1,  \ldots, v_{n-1}),
$
with each generator $v_i$ only defined modulo the previous ones.
\end{remark}

The notion of the height of a formal group is fundamental to the geometry of the stack $\mathcal M^\heart_\mathrm{FG}$. This is the content of the following theorem, which is in effect only a reinterpretation of Landweber and Morava's invariant prime ideal theorem \cite[Theorem 3.3.6]{Orange Book}. But for a statement in terms of stacks and a direct proof, see \cite[Theorem 5.13]{Goerss}.

\begin{theorem}[Landweber, Morava]
Ranging over all\footnote{Here  $\mathcal M_{\mathrm{FG}}^{\heart, \ge \infty} =\bigcap_{n\ge 1} \mathcal M_{\mathrm{FG}}^{\heart, \ge n}$ , is the ordinary stack of formal groups of infinite height (e.g.\ the additive formal group). This stack has some rather unsavory properties (e.g.\ it is not geometric) and does not fit so nicely into the chromatic paradigm. As such, we will be disregarding it in all remaining discussions.} $1\le  n\le \i$,
the collection of closed substacks $\mathcal M_{\mathrm{FG}}^{\heart, \ge n}\subseteq\mathcal M^\heart_\mathrm{FG}\o_{\mathbf Z}\mathbf Z_{(p)}$ of formal groups of height $\ge n$, is the complete list of reduced closed substacks of the ordinary stack of formal groups $\mathcal M^\heart_\mathrm{FG}\o_{\mathbf Z}\mathbf Z_{(p)}$.
\end{theorem}

\begin{definition} The \textit{stack of formal groups of height $\le n$} is defined to be the open complement
$\mathcal M^{\heart, \le n}_\mathrm{FG} :=\mathcal M_\mathrm{FG}^\heart\o_{\mathbf Z}\mathbf Z_{(p)} - \mathcal M_\mathrm{FG}^{\heart, \ge n+1}$   in $\mathcal M_\mathrm{FG}^\heart\o_{\mathbf Z}\mathbf Z_{(p)}$ of the reduced closed substack $\mathcal M_\mathrm{FG}^{\heart, \ge n+1}$ of formal groups of height $\ge n+1$. 
\end{definition}

The ordinary stack $\mathcal M^{\heart, \le n}_\mathrm{FG}$ (as always implicitly also depending on the choice of $p$) by definition classifies formal groups \textit{of height $\le n$} - for an equivalent definition of the latter, see \cite[Definition 4.4.16.]{Elliptic 2}.
We thus obtain  a decreasing chain of closed substacks 
$$
\mathcal M_\mathrm{FG}^\heart\o_{\mathbf Z}\mathbf Z_{(p)}\supseteq
\mathcal M_\mathrm{FG}^{\heart, \ge 1}\supseteq
\mathcal M_\mathrm{FG}^{\heart, \ge 2}\supseteq
\cdots \supseteq
\mathcal M_\mathrm{FG}^{\heart, \ge n-1}\supseteq
\mathcal M_\mathrm{FG}^{\heart, \ge n}\supseteq
\cdots
,
$$
and a complementary increasing chain of open substacks
\begin{equation}\label{height filtration on heart}
\mathcal M^{\heart, \le 1}_\mathrm{FG}\subseteq
\mathcal M^{\heart, \le 2}_\mathrm{FG}\subseteq
\cdots\subseteq
\mathcal M^{\heart, \le n-1}_\mathrm{FG}\subseteq
\mathcal M^{\heart, \le n}_\mathrm{FG}\subseteq
\cdots\subseteq
\mathcal M^\heart_\mathrm{FG}\o_{\mathbf Z}\mathbf Z_{(p)}.
\end{equation}

\begin{definition}
The \textit{stack of formal groups of exact height $n$} is defined to be the stratum of either of the above filtrations, which is to say as
\begin{eqnarray*}
\mathcal M^{\heart, =n}_\mathrm{FG} &:=& \mathcal M^{\heart,\ge n}_\mathrm{FG}\times_{\mathcal M_\mathrm{FG}^\heart\o_{\mathbf Z}\mathbf Z_{(p)}}\mathcal M^{\heart,\le n}_\mathrm{FG} \\
&\simeq& \mathcal M^{\heart, \le n}_\mathrm{FG} - \mathcal M^{\heart, \le n-1}_\mathrm{FG} \\
&\simeq& \mathcal M^{\heart, \ge n}_\mathrm{FG} - \mathcal M^{\heart, \ge n+1}_\mathrm{FG},
\end{eqnarray*}
and is a reduced closed substack of $\mathcal M^{\heart, \le n}_\mathrm{FG}$.
\end{definition}

\begin{remark}\label{Landscription}
As explained in \cite[Remark 4.4.18]{Elliptic 2},
in terms of the Landweber ideals, a formal group $\w{\G}$ over a commutative ring $R$ has height $\le n$ if $\mathfrak I^{\w{\G}}_{n+1} = R,$ and has exact height $n$ if additionally $\mathfrak I^{\w{\G}}_{n} = (0)$.
\end{remark}

\begin{remark}\label{def height n}
Though the complement formula $\mathcal M^{\heart, =n}_\mathrm{FG} = \mathcal M^{\heart, \le n}_\mathrm{FG} - \mathcal M^{\heart, \le n-1}$ holds by definition on the level of stacks, it is not necessarily true that the the canonical inclusion $\mathcal M^{\heart, =n}_\mathrm{FG}(R)\subseteq\mathcal M^{\heart, \le n}_\mathrm{FG}(R)-\mathcal M_\mathrm{FG}^{\heart, \le n-1}(R)$ is bijective for any commutative ring $R$ (though it is the case if $R$ is a field). That is to say, there might be formal groups $\w{\G}$ over $R$ which are of height $\le n$, not of height $\le n-1$, and yet do not have exact height $n$. We say that such formal groups \textit{have height $n$}. In terms of Landweber ideals, height $n$ means that $\mathfrak I_{n+1}^{\w{\G}}=R$ and $\mathfrak I _n^{\w{\G}}\subsetneq R$, whereas for exact height $n$, the second requirement is $\mathfrak I^{\w{\G}}_{n} = (0)$ instead.
\end{remark}

For every $n\ge 1$, over a separably closed field $\kappa$ of characteristic $p$, there exits by Lazard's Classification Theorem \cite[Lecture 13, Theorem 10]{Lurie Chromatic} or  \cite[Corollary 15.4]{Pstragowski} a height $n$-formal group $\w{\G}_0$, and it is unique up to isomorphism. That is to say, the stack $\mathcal M_\mathrm{FG}^{\heart, =n}$ has a single geometric point. This allows us by \cite[Theorem 4.3.8]{Smithling}, {\cite[Lecture 19, Proposition 1]{Lurie Chromatic}},  \cite[Theorem 17.9]{Pstragowski} or \cite[5.36 Theorem]{Goerss} to exhibit the stack of formal groups of exact height $n$ as a quotient stack in $\kappa$-schemes under the action of the automorphism group scheme $\underline{\mathrm{Aut}}(\w{\G}_0) = \Spec(\kappa)\times_{\mathcal M_{\mathrm{FG}}^\heart}\Spec(\kappa)$:

\begin{theorem}\label{orbit picture at height n}
A map of ordinary stacks $\Spec(\overline{\mathbf F}_p)\to\mathcal M^{\heart, =n}_\mathrm{FG}$, classifying a formal group $\w{\G}_0$ of exact height $n$ over $\overline{\mathbf F}_p$, induces an equivalence $\mathcal M_\mathrm{FG}^{\heart, =n}\simeq \Spec(\overline{\mathbf F}_p)/\underline{\mathrm{Aut}}(\w{\G}_0)$.
\end{theorem}

This concludes the review of the theory of heights of formal groups that we will require.

\subsection{The moduli stack of oriented formal groups of height $\le n$}

Let us now use the decompression mechanism of Proposition \ref{Decompressing opens} to obtain an open substack of the non-connective spectral stack $\M\o  S_{(p)}$ from  the open substack $\mathcal M^{\heart, \le n}_\mathrm{FG}\subseteq\mathcal M^\heart_\mathrm{FG}\o_{\mathbf Z}\mathbf Z_{(p)}$.

\begin{definition}\label{Def of MFGorn}
For any prime $p$ and integer $n\ge 1$, the \textit{moduli stack of oriented formal groups of height $\le n$} is defined to be the fiber product
$$
\begin{tikzcd}
\mathcal M^{\mathrm{or},\le n}_\mathrm{FG}\ar{r} \ar{d} & \mathcal M^{\heart,\le n}_\mathrm{FG}\circ\pi_0\ar{d}\\
\M \o  S_{(p)}\ar{r} & (\mathcal M^{\heart}_\mathrm{FG}\o_{\mathbf Z}\mathbf Z_{(p)})\circ\pi_0
\end{tikzcd}
$$
in the $\i$-category $\Shv_\mathrm{fpqc}^\mathrm{nc}$.
\end{definition}

\begin{remark}\label{expldef of len}
Note that for any $\E$-ring $A$, the universal $\E$-ring map $S\to A$ factors through $S\to S_{(p)}$ if and only if $A$ is $(p)$-local.
Together with {\cite[Proposition 4.3.23]{Elliptic 2}}, this  gives rise to an explicit description 
 of the functor $\M :\CAlg\to\mS$ as
$$
(\M\o  S_{(p)}) (A) \simeq
\begin{cases}
* & \text{if $A$ is a $(p)$-local complex periodic $\E$-ring.}\\
\emptyset & \text{otherwise,}
\end{cases}
$$
It follows that the pullback description of $\mathcal M^{\mathrm{or}, \le n}_\mathrm{FG}$ from Definition \ref{Def of MFGorn} determines its value on any $\E$-ring $A$ to be
$$
\mathcal M^{\mathrm{or}, \le n}_{\mathrm{FG}}(A) \simeq
\begin{cases}
* & \text{if $A$ is $(p)$-local complex periodic, and $\w{\G}{}^{\CMcal Q_0}_A$ is of height $ \le n$,}\\
\emptyset & \text{otherwise.}
\end{cases}
$$
The canonical map of non-connective spectral stacks $\mathcal M^{\mathrm{or}, \le n}_\mathrm{FG}\to \M\o  S_{(p)}$ hence amounts to the inclusion of $(p)$-local complex periodic $\E$-rings, whose classical Quillen formal groups are of height $\le n$, (or equivalently by Remark \ref{Landscription}, which satisfy $\mathfrak I^A_{n+1} =\pi_0(A)$,) into all $(p)$-local complex periodic $\E$-rings.
\end{remark}

\begin{lemma}\label{Lemma 1-line}
The canonical map $j_n:\mathcal M^{\mathrm{or}, \le n}_\mathrm{FG}\to \M\o  S_{(p)}$ is an open immersion of non-connective spectral stacks. On underlying ordinary stacks, it induces the open immersion $\mathcal M^{\heart, \le n}_\mathrm{FG}\subseteq \mathcal M_{\mathrm{FG}}^\heart\o_{\mathbf Z}\mathbf Z_{(p)}$ of ordinary formal groups of height $\le n$ into all formal groups (in the $(p)$-local setting).
\end{lemma}

\begin{proof}
This is a direct application of  Proposition \ref{Decompressing opens}.
\end{proof}

The next few results are dedicated to  obtaining a better understanding of the non-connective spectral stack $\mathcal M^{\mathrm{or}, \le n}_\mathrm{FG}$. 
They accomplishing for it what  \cite[Theorem 2.3.1 - Corollary 2.3.7]{ChromaticCartoon} do for $\M$.

\begin{prop}\label{Landweber cover}
Let $E$ be an even periodic $(p)$-local Landweber exact $\E$-ring, for which the formal group $\w{\G}{}^{\CMcal Q_0}_A$ has height $n$, in the sense of Remark \ref{def height n}. The essentially unique map of non-connective spectral stacks $\Spec(E)\to \mathcal M^{\mathrm{or}, \le n}_\mathrm{FG}$ is faithfully flat.
\end{prop}

\begin{proof}
For any $\E$-ring $A$ with a map $\Spec(A)\to\mathcal M^{\mathrm{or}, \le n}_\mathrm{FG}$, we must show that the map $\Spec(E)\times_{\mathcal M^{\mathrm{or}, \le n}_\mathrm{FG}}\Spec(A)\to \Spec(A)$ is faithfully flat. It follows from the explicit description of the functor $\mathcal M^{\mathrm{or}, \le n}_\mathrm{FG}$ in Remark \ref{expldef of len} that $A$ is complex periodic, and that the fiber product is $\Spec(E)\times_{\mathcal M^{\mathrm{or}, \le n}_\mathrm{FG}}\Spec(A)\simeq \Spec(E\o A).$ Consequently, we need to verify that the map of $\E$-rings $A\to E\o A$ is faithfully flat.

By assumption, $E$ is a Lansweber exact even periodic spectrum. That means that it is, as a homology theory (on spectra), obtained from periodic complex bordism  $\MP$ as
$$
\pi_n(E\o X) = E_n(X)\simeq \pi_0(E)\o_{\pi_0(\MP)} \MP_n(X)\simeq \pi_0(E)\o_{\pi_0(\MP)} \pi_n(\MP\o X)
$$
for any spectrum $X$. Plugging in $X=A$, and since we already know from \cite[Lemma 2.3.5]{ChromaticCartoon} that the $\E$-ring map $A\to\MP\o A$ is flat, it follows that
$$
\pi_n(\MP\o A)\simeq \pi_0(\MP\o A)\o_{\pi_0(A)} \pi_n(A)
$$
and consequently
\begin{eqnarray*}
\pi_n(E\o A)&\simeq& \pi_0(E)\o_{\pi_0(\MP)} (\pi_0(\MP\o A)\o_{\pi_0(A)} \pi_n(A))\\
&\simeq& (\pi_0(E)\o_{\pi_0(\MP)} \pi_0(\MP\o A))\o_{\pi_0(A)} \pi_n(A)\\
&\simeq& \pi_0(E\o A)\o_{\pi_0(A)} \pi_n(A),
\end{eqnarray*}
as desired. It remains for us to verify that $\pi_0(A)\to \pi_0(E\o A)$ is faithfully flat. For a direct proof, see \cite[Lemma 21.8]{Pstragowski}, but the equivalent claim that the map of ordinary stacks $\Spec(\pi_0(E))\to \mathcal M^{\heart, \le n}_\mathrm{FG}$ is fully faithful is one of the main results of \cite{Naumann}.
\end{proof}

\begin{corollary}
The non-connective spectral stack $\mathcal M^{\mathrm{or}, \le n}_\mathrm{FG}$ is geometric.
\end{corollary}

\begin{proof}
Since the affineness of the diagonal for  $\mathcal M^{\mathrm{or}, \le n}_\mathrm{FG}$ is clear by a completely analogous argument to \cite[Lemma 2.3.2]{ChromaticCartoon}, this follows directly from Proposition in light of the definition of geometric non-connective spectral stacks \cite[Definition 1.3.1]{ChromaticCartoon}.
\end{proof}

\begin{corollary}\label{Cech nerve for len}
There is a canonical equivalence
$$
\mathcal M^{\mathrm{or}, \le n}_\mathrm{FG}\simeq \left|\mathrm{\check C}^\bullet(\Spec(E)/\Spec(S))\right|,
$$
with the geometric realization formed in  the $\i$-topos $\mathcal S\mathrm{hv}_{\mathrm{fpqc}}^\mathrm{nc}$.
\end{corollary}

\begin{proof}
In light of Proposition \ref{Landweber cover}, this is an instance of \cite[Proposition 1.3.5]{ChromaticCartoon}, or more specifically, an application of \cite[Lemma 1.3.6]{ChromaticCartoon}.
\end{proof}

\begin{remark}\label{Cech nerve for len with LT}
One $\E$-ring satisfying the assumptions of Proposition \ref{Landweber cover} is the Lubin-Tate spectrum $E_n$ (associated to any height $n$ formal group over any perfect field of characteristic $p$). Corollary \ref{Cech nerve for len} therefore asserts the colimit presentation
$$
\xymatrix{
\mathcal M^{\mathrm{or}, \le n}_\mathrm{FG}\simeq \varinjlim\Big( \cdots  \ar@<-1.5ex>[r]\ar@<-.5ex>[r] \ar@<.5ex>[r] \ar@<1.5ex>[r]& \Spec(E_n\o E_n\o E_n)\ar@<-1ex>[r]\ar[r] \ar@<1ex>[r] &\Spec(E_n\o E_n)\ar@<-.5ex>[r]\ar@<.5ex>[r] & \Spec(E_n)\Big )
}
$$
for the moduli stack of oriented formal groups of height $\le n$.
\end{remark}

\subsection{Quasi-coherent sheaves on $\mathcal M^{\mathrm{or}, \le n}_\mathrm{FG}$ and chromatic localizations}

The functoriality of quasi-coherent sheaves, discussed in \cite[Section 1.4]{ChromaticCartoon} therefore gives rise to adjunctions
$$
\xymatrix{
\Sp_{(p)} \ar@<.5ex>[r]^{p^*\qquad\quad}&\QCoh(\M\o  S_{(p)}) \ar@<.5ex>[l]^{p_*\qquad\quad} \ar@<.5ex>[r]^{\quad j_n^*}& \QCoh(\mathcal M^{\mathrm{or}, \le n}_\mathrm{FG}) \ar@<.5ex>[l]^{\quad (j_n)_*},
}
$$
where $p:\M\o  S_{(p)}\to\Spec(S_{(p)})$ is the canonical map, and where we have used the canonical equivalence of $\i$-categories $\QCoh(\Spec(S_{(p)})\simeq \Sp_{(p)}$. The composite adjunction is just (the  $(p)$-local version of) the adjunction $\sO_{\mathcal M^{\mathrm{or}, \le n}_\mathrm{FG}}\o (-)\dashv \Gamma(\mathcal M^{\mathrm{or}, \le n}_\mathrm{FG};-)$ of Example \cite[Example 1.4.4]{ChromaticCartoon}, and we will often use this notation for it.

\begin{theorem}\label{Chromatic localization is restriction}
The adjunction
$$
\sO_{\mathcal M_{\mathrm{FG}}^{\mathrm{or}, \le n}}\o (-) :\Sp_{(p)}\rightleftarrows \QCoh(\mathcal M^{\mathrm{or}, \le n}_\mathrm{FG}) :\Gamma(\mathcal M^{\mathrm{or}, \le n}_\mathrm{FG}; -)
$$
of the preceding discussion exhibits an equivalence of  symmetric monoidal $\i$-categories
$
\QCoh(\mathcal M^{\mathrm{or}, \le n}_\mathrm{FG})\simeq L_n\mathrm{Sp}
$
between quasi-coherent sheaves on $\mathcal M^{\mathrm{or}, \le n}_\mathrm{FG}$ and $E(n)$-local spectra.
\end{theorem}

\begin{proof}[Proof of Theorem \ref{Chromatic localization is restriction}]
According to Lemma \ref{Cech nerve for len}, or more precisely Remark \ref{Cech nerve for len with LT}, there is a canonical equivalence of $\i$-categories
$\QCoh(\mathcal M_\mathrm{FG}^{\mathrm{or}, \le n})\simeq \Tot\big( \Mod_{E_n^{\otimes (\bullet +1)}} \big),$ where the right-hand side is obtained by totalization from the Amitsur complex for $E_n$. In the language of \cite[Definition 3.18]{Mathew}, the Hopkins-Ravenel Smash Product Theorem is equivalent to saying that the canonical $\E$-ring map $L_nS\to E_n$ is descendable, see \cite[Theorem 4.18]{Mathew}. Therefore standard descent theory, e.g.\ \cite[Proposition 3.22]{Mathew} or \cite[Lemma D.3.5.8]{SAG}, shows that the canonical map $L_n\mathrm{Sp}\to \Tot\big( \Mod_{E_n^{\otimes (\bullet +1)}} \big)$ is an equivalence of $\i$-categories.

Consequently, we obtain an equivalence of $\i$-categories $\QCoh(\mathcal M^{\mathrm{or}, \le n}_\mathrm{FG})\simeq L_n\mathrm{Sp}$. By tracing through all the identifications, one finds that it is induced by push-pull functoriality from the terminal map $p \circ j_n :\mathcal M^{\mathrm{or}, \le n}_\mathrm{FG}\to \Spec(S)$, seeing how the latter factors through $\Spec(L_nS)\to\Spec(S)$. That is to say, the preceding argument shows that the adjunction
$$
p_n^*:\QCoh(\mathcal M^{\mathrm{or}, \le n}_\mathrm{FG})\rightleftarrows \QCoh(\Spec(L_nS))\simeq L_n\mathrm{Sp}:(p_n)_*,
$$
induced by the essentially unique map $p_n:\mathcal M^{\mathrm{or}, \le n}_\mathrm{FG}\to\Spec(L_nS)$, 
is an adjoint equivalence of $\i$-categories. Because the localization $L_n$ is smashing, the equivalence of $\i$-categories $\QCoh(\Spec(L_nS))\simeq L_n\mathrm{Sp}$ is symmetric monoidal. Now the symmetric monoidality claim follows from the evident symmetric monoidality of the quasi-coherent pullback $p_n^*$.
\end{proof}

\begin{remark}
Let us indicate a slight reformulation of the proof of  Theorem \ref{Chromatic localization is restriction} (though both fundamentally boil down to applying the Barr-Beck-Lurie Comonadicity Theorem and the Hopkins-Ravenel Smash Product Theorem). From Lemma \ref{Cech nerve for len} and the Beck-Chevalley comonadic formulation of descent, we may conclude that there is a canonical equivalence of $\i$-categories
$$
\QCoh(\mathcal M^{\mathrm{or}, \le n}_\mathrm{FG})\simeq\cMod_{E\o E}(\Mod_E).
$$
Here the comodule $\i$-category on the right is interpreted in the sense of \cite[Construction 2.4.8]{ChromaticCartoon} or equivalently \cite{Torii}. Then we may apply \cite[Theorem 9]{Torii} to identify the comodule $\i$-category in question with the $\i$-category of $E$-local spectra, which is the same as $E(n)$-local spectra as consequence of \cite[Corollary 1.12]{Hovey}.
\end{remark}

\begin{corollary}\label{L_n as restriction}
For any $(p)$-local spectrum $X$, its $n$-th chromatic localization $L_nX$ is, in terms of the equivalence of $\i$-categories of Theorem \ref{Chromatic localization is restriction}, given by the constant sheaf
$\sO_{\mathcal M^{\mathrm{or}, \le n}_\mathrm{FG}}\o X,$ or equivalently as
the sheaf restriction $(\sO_{\M}\o X)\vert_{\mathcal M^{\mathrm{or}, \le n}_\mathrm{FG}}.$ Equivalently yet, it is given explicitly in terms of spectra by
$$
L_nX\simeq \Gamma(\mathcal M^{\mathrm{or}, \le n}_\mathrm{FG}; \sO_{\M}\o X).
$$
\end{corollary}

\begin{proof}
This follows directly from Theorem \ref{Chromatic localization is restriction}. Alternatively, it may be seen directly by noting that Remark \ref{Cech nerve for len with LT} gives rise to the formula
$$
\Gamma(\mathcal M_\mathrm{FG}^{\mathrm{or}, \le n}; \sO_{\mathcal M_\mathrm{FG}^{\mathrm{or}}}\o X)\simeq  \Tot\big(E_n^{\o (\bullet+1)}\o X\big).
$$
The right-hand side is the $E_n$-nilpotent completion of $X$, and its equivalence with the Bousfield localization $L_{E_n}X\simeq L_nX$ for any spectrum $X$ is an equivalent form of the Smash Product Theorem, see \cite[Lecture 31]{Lurie Chromatic}.
\end{proof}

\begin{remark}
Let $\mC\subseteq\CAlg$ denote the full subcategory of complex periodic $(p)$-local $\E$-rings, and let $\mC_{\le n}\subseteq\mC$ be the full subcategory of those $A$ for which the ordinary Quillen formal group $\w{\G}{}^{\CMcal Q_0}_A$ is of height $\le n$. Then we have colimit formulas
$
\M\simeq \varinjlim_{A\in \mC}\Spec(A),
$
and
$\mathcal M^{\mathrm{or}, \le n}_\mathrm{FG}\simeq\varinjlim_{A\in \mC_{\le n}}\Spec(A)\simeq \varinjlim_{A\in \mC}\Spec(L_{\mathfrak I{}^{n+1}_A}(A)),$
 hence Theorem \ref{Chromatic localization is restriction}
exhibits the $n$-th chromatic localization of the $\i$-category of spectra  as the limit of $\i$-categories
$$
L_n\mathrm{Sp}\simeq \varprojlim_{A\in \mC}\Mod_A^{\mathrm{Loc}(\mathfrak I_{n+1}^{A})}.
$$
Analogously, Corollary \ref{L_n as restriction} presents chromatic localization of a spectrum $X$ as
$$
L_nX\simeq \varprojlim_{A\in \mC} L_{\mathfrak I_{n+1}^A}(A\o X).
$$
\end{remark}

\begin{remark}
It follows from Corollary \ref{L_n as restriction} that a $(p)$-local spectrum $X$ is $E(n)$-local if and only if the quasi-coherent sheaf $\sO_{\M}\o X\in\QCoh(\M\o  S_{(p)})$ belongs to the full subcategory $\QCoh(\mathcal M_\mathrm{FG}^{\mathrm{or}, \le n})\subseteq\QCoh(\M\o  S_{(p)})$. 
 This weaker result may be proved directly, without appeal to Theorem \ref{Chromatic localization is restriction}, and without appeal to the Smash Product Theorem. Indeed,
by Proposition \ref{semi-orthog decomp}, quasi-coherent sheaves on $\M\o S_{(p)}$ semi-orthogonally decompose into quasi-coherent sheaves supported along $\mathcal M^{\heart, \ge n+1}_\mathrm{FG}\subseteq\mathcal M^{\heart}_\mathrm{FG}$, and the quasi-coherent sheaves on the open complement $\mathcal M^{\mathrm{or}, \le n}_\mathrm{FG}$. Similarly, $(p)$-local spectra semi-orthogonally decompose into $E(n)$-acyclic and $E(n)$-local spectra. Since the functor $X\mapsto \sO_{\M}\o X$ preserves cofiber sequences, it therefore suffices to show that a $(p)$-local spectrum is $E(n)$-acyclic if and only if the quasi-coherent sheaf $\sF =\sO_{\M}\o X$ is supported along $\mathcal M^{\heart, \ge n+1}_\mathrm{FG}$.
 The latter condition is equivalent to demanding that $\sF\vert_{\mathcal M^{\mathrm{or}, \le n}_\mathrm{FG}}\simeq 0$ in $\QCoh(\M\o  S_{(p)})$, or equivalently that $\pi_k(\sF)\vert_{\mathcal M^{\heart, \le n}_\mathrm{FG}}\simeq 0$ in $\QCoh(\mathcal M^\heart_\mathrm{FG})^\heart$ for all $k\in \mathbf Z$. Recall from \cite[Remark 2.5.4]{ChromaticCartoon} that $\pi_k(\sF)\simeq \sF_k(X)$ are the usual (ordinary) quasi-coherent sheaves on $\mathcal M^{\heart}_\mathrm{FG}$, associated to the spectrum $X$, as appearing in usual stacky approaches to chromatic homotopy theory as developed in \cite{COCTALOS}, \cite{Pstragowski}, \cite{Lurie Chromatic} (denoted in the latter $\sF_{\Sigma^kX}$). In particular, the desired claim now boils down to the fact \cite[Lecture 22]{Lurie Chromatic} that the spectrum $X$ being $E(n)$-acyclic is equivalent to the sheaves $\sF_k(X)$ being supported on the reduced closed substack $\mathcal M^{\heart, \ge n+1}_\mathrm{FG}\subseteq\mathcal M^{\heart}_\mathrm{FG}$ for all $k\in \mathbf Z$ (sufficiently, for $k=0, 1$).
\end{remark}

\begin{remark}\label{IndCoh interp}
Under a very artificial definition of ind-coherent sheaves \cite[Definition 2.4.2]{ChromaticCartoon}, a $(p)$-local analogue of \cite[Theorem 2.4.4]{ChromaticCartoon} shows that the adjunction $p^*\dashv p_*$ induces an equivalence of $\i$-categories $\IndCoh(\M\o  S_{(p)})\simeq \Sp_{(p)}$. In light of that, Theorem \ref{Chromatic localization is restriction} may be seen as combining two equivalences of $\i$-categories. The first and easy one, following from the discussion in the preceding Remark, asserts that (under appropriate definition of ind-coherent sheaves, which we do not spell out here) we have $\IndCoh(\mathcal M^{\mathrm{or},\le n}_\mathrm{FG})\simeq L_n\Sp$. The second and much harder statement is that the canonical functor $\IndCoh(\mathcal M^{\mathrm{or},\le n}_\mathrm{FG})\to \QCoh(\mathcal M^{\mathrm{or},\le n}_\mathrm{FG})$ is an equivalence of $\i$-categories. Since restriction to quasi-compact opens is easily seen to be smashing on quasi-coherent sheaves, this  amounts roughly to the assertion of the Smash Product Theorem.
\end{remark}

 Any quasi-coherent sheaf $\sF\in\QCoh(\M\o  S_{(p)}$ gives rise by Corollary \ref{corollary semi-orthogonal sequence} to a cofiber sequence
\begin{equation}\label{sheaf fiber sequence}
\Gamma_{\mathcal M^{\heart, \ge n+1}_\mathrm{FG}}(\sF)\to \sF\to \sF\vert_{\mathcal M^{\mathrm{or}, \le n}_\mathrm{FG}}.
\end{equation}
Similarly, the yoga of Bousfield localization implies for any $(p)$-complete spectrum the existence of a cofiber sequence
\begin{equation}\label{chromatic fiber sequence}
C_n X\to X\to L_n X
\end{equation}
in $\Sp_{(p)}$, where $C_nX$ is the $E(n)$-acylization of $X$, see for instance \cite[Lecture 32]{Lurie Chromatic}.

\begin{corollary}\label{C_n as local cohomology}
For any finite $(p)$-local spectrum $X$, the above cofiber sequence \eqref{chromatic fiber sequence} in $\Sp$ is obtained from the fiber sequence \eqref{sheaf fiber sequence} in $\QCoh(\M\o_S S_{(p)})$ by passage to global sections. In particular, there is a canonical equivalence of spectra
$$
C_nX\simeq \Gamma_{\mathcal M^{\heart, \ge n+1}_\mathrm{FG}}(\M\o S_{(p)}; \sO_{\M}\o X).
$$
\end{corollary}

\begin{proof}
By the evident $(p)$-local version of \cite[Lemma 2.4.5]{ChromaticCartoon}, the canonical map $X\to\Gamma(\M\o S_{(p)}; \sO_{\M}\o X)$ is an equivalence of spectra. Together with Corollary \ref{L_n as restriction}, this proves that applying the functor $\Gamma(\M\o S_{(p)};-)$ to the cofiber sequence \eqref{sheaf fiber sequence} reproduces the cofiber sequence \eqref{chromatic fiber sequence}. We obtain the equivalence in the statement of the Corollary by comparing the first terms of the two cofiber sequences in question.
\end{proof}

\begin{remark}
For a not-necessarily-finite $(p)$-local spectrum $X$, it follows just like in the proof of \cite[Proposition 2.4.1]{ChromaticCartoon} that $\Gamma(\M\o S_{(p)}; \sO_{\M}\o X)\simeq X^\wedge_{\MP}$ is the $\MP$-nilpotent completion of $X$. Global sections of the constant sheaf $\sO_{\M}\o X$, supported along the stack of formal groups of height $\ge n+1$, may therefore always be expressed as
$$\Gamma_{\mathcal M^{\heart, \ge n+1}_\mathrm{FG}}(\M\o S_{(p)}; \sO_{\M}\o X)\simeq C_nX^\wedge_\mathrm{MP}.$$
Here the order in which we apply $\MP$-nilpotent completion and $E(n)$-acyclization does not matter, since the functor $C_n=\mathrm{fib}(\mathrm{id}\to L_n)$ commutes with both limits and smash products as consequence of the Smash Product Theorem.
\end{remark}

\begin{remark}As we saw in the last Remark,
unlike Corollary \ref{L_n as restriction}, the conclusion of Corollary \ref{C_n as local cohomology} genuinely requires the finiteness assumption on $X$. In the spirit of Remark \ref{IndCoh interp}, we might informally assert that the functor of ind-coherent global sections induces an equivalence  $\IndCoh_{\mathcal M^{\heart, \ge n+1}_\mathrm{FG}}(\M\o S_{(p)})\simeq C_n(\Sp)$ with the $\i$-category of $E(n)$-acyclic spectra. But unlike in  the case of Theorem \ref{Chromatic localization is restriction} and  Remark \ref{IndCoh interp}, the canonical comparison functor $\IndCoh_{\mathcal M^{\heart, \ge n+1}_\mathrm{FG}}(\M\o S_{(p)})\to \QCoh_{\mathcal M^{\heart, \ge n+1}_\mathrm{FG}}(\M\o S_{(p)})$ need not be an equivalence of $\i$-categories.
\end{remark}

\subsection{The height filtration on $\M$ and the chromatic filtration}
It follows from Definition  \ref{Def of MFGorn} that the  height filtration \eqref{height filtration on heart} of the ordinary stack $\mathcal M^\heart_\mathrm{FG}\o_{\mathbf Z}\mathbf Z_{(p)}$ induces by decompression, i.e.\ pullback along the compression map $\M\to \mathcal M^\heart_\mathrm{FG}\circ \pi_0,$ a tower of non-connective spectral stacks
$$
\mathcal M^{\mathrm{or}, \le 1}_\mathrm{FG}\subseteq
\mathcal M^{\mathrm{or}, \le 2}_\mathrm{FG}\subseteq
\cdots\subseteq
\mathcal M^{\mathrm{or}, \le n-1}_\mathrm{FG}\subseteq
\mathcal M^{\mathrm{or}, \le n}_\mathrm{FG}\subseteq
\cdots\subseteq
\M\o S_{(p)}.
$$
For any quasi-coherent sheaf $\sF$ on $\M\o S_{(p)}$, this induces through restriction a filtration
$$
\sF\to
\cdots\to
\sF\vert_{\mathcal M^{\mathrm{or}, \le n}_\mathrm{FG}}\to
\sF\vert_{\mathcal M^{\mathrm{or}, \le n-1}_\mathrm{FG}}\to
\cdots \to
\sF\vert_{\mathcal M^{\mathrm{or}, \le 2}_\mathrm{FG}}\to
\sF\vert_{\mathcal M^{\mathrm{or}, \le 1}_\mathrm{FG}}.
$$
inside the $\i$-category $\QCoh(\M\o S_{(p)})$. We call this the \textit{chromatic filtration} on $\sF$, justified by the next observation:

\begin{prop}\label{Chromatic filtration = chromatic filtration}
Let $X$ be a $(p)$-local spectrum.
The chromatic filtration in the above sense,  on the quasi-coherent sheaf $\sF = \sO_{\M}\o X$ in the $\i$-category $\QCoh(\M\o S_{(p)})$, induces upon passage to global sections the chromatic filtration
$$
X\to
\cdots\to
L_nX\to
L_{n-1}X\to
\cdots\to
L_2 X\to
L_1X
$$
in the usual sense,
on $X$ in the $\i$-category of spectra.
\end{prop}

\begin{proof}
This is a direct consequence of Corollary \ref{L_n as restriction}.
\end{proof}

\begin{remark}\label{moduli stack of finite height}
Let $\mathcal M^{\mathrm{or}, < \infty}_\mathrm{FG} := \varinjlim_n \mathcal M^{\mathrm{or}, \le n}_\mathrm{FG}$ denote the colimit of the height filtration, i.e.\ the moduli stack of oriented formal groups of finite height. Then Proposition \ref{Chromatic filtration = chromatic filtration} and Theorem \ref{Chromatic localization is restriction} together imply that the global sections functor $\QCoh(\mathcal M^{\mathrm{or}, < \infty}_\mathrm{FG})\to\Sp_{(p)}$ is fully faithful, and its essential image consists of the chromatically complete (\textit{harmonic}, in the language of \cite{Ravenel:Bousfield}) $(p)$-local spectra . In the spirit of Remark \ref{IndCoh interp}, we may informally say that the canonical comparison functor $\IndCoh(\mathcal M^{\mathrm{or}, < \infty}_\mathrm{FG})\to\QCoh(\mathcal M^{\mathrm{or}, < \infty}_\mathrm{FG})$ is an equivalence of $\i$-categories. The difference between the $\i$-categories $\QCoh(\M\o S_{(p)})$ and $\IndCoh(\M\o S_{(p)})\simeq \Sp_{(p)}$ is therefore something that occurs at infinite height.
\end{remark}

Even though the height filtration on the stack of formal groups (oriented or otherwise) is not exhaustive, we now show that the induced chromatic filtration on quasi-coherent sheaves turns out to be, at least under certain finiteness assumptions.

\begin{remark}
To specify these finiteness assumptions, we will say that a quasi-coherent sheaf on a geometric ordinary stack is \textit{coherent}, if a pullback of it to an fpqc cover by an affine scheme is a coherent sheaf in the usual sense. Due to coherence of the Lazard ring, this will be equivalent in the case in question to the assumption of finite presentation - see {\cite[Remark 1.2]{Goerss}}.
\end{remark}

\begin{theorem}[Chromatic convergence]\label{CCT}
Let $\sF$ be a quasi-coherent sheaf on $\M\o S_{(p)}$, and assume that $\pi_t(\sF)$ is a coherent sheaf on the underlying ordinary stack $\mathcal M^\heart_\mathrm{FG}$ for every $t\in \mathbf Z$ (equivalently, for $t=0, 1$). Then the chromatic filtration on $\sF$ is complete, in the sense that it exhibits an equivalence of quasi-coherent sheaves
$
\sF\simeq \varprojlim_n \sF\vert_{\mathcal M^{\mathrm{or}, \le n}_\mathrm{FG}}.
$
\end{theorem}

\begin{proof}
In light of the cofiber sequences \eqref{sheaf fiber sequence}, we have
$$
\fib(\sF\to \varprojlim_n\sF\vert_{\mathcal M^{\mathrm{or}, \le n}_\mathrm{FG}})\simeq \varprojlim_n \fib(\sF\to \sF_{\mathcal M^{\mathrm{or}, \le n}_\mathrm{FG}})\simeq \varprojlim_n \Gamma_{\mathcal M^{\heart, \ge n+1}_\mathrm{FG}}(\sF).
$$
We must therefore show that the limit of the tower $(\Gamma_{\mathcal M^{\heart, \ge n}_\mathrm{FG}}(\sF))_{n}$ vanishes.
Toward that goal, recall from
Proposition \ref{lc spectral sequence} the local cohomology spectral sequences
$$
E\langle n\rangle^{s, t}_2 = \mathcal H^s_{\mathcal M^{\heart, \ge n}_\mathrm{FG}}(\pi_t(\sF))\Rightarrow \pi_{t-s}(\Gamma_{\mathcal M^{\heart, \ge n}_\mathrm{FG}}(\sF)).
$$
functorial in $n$, as it traverses the chromatic filtration.

We claim that the filtered system $(E\langle n\rangle_2^{s,t})_n  = (\mathcal H^s_{\mathcal M^{\heart, \ge n}_\mathrm{FG}}(\pi_t(\sF)))_n$ is pro-isomorphic to zero. For a fixed $t$, this follows from the algebraic chromatic convergence theorem of \cite{Goerss}, or more precisely the auxiliary result \cite[Theorem 8.20]{Goerss}. That asserts that the map $\mathcal H^s_{\mathcal M^{\heart, \ge n}_\mathrm{FG}}(\pi_t(\sF))\to \mathcal H^s_{\mathcal M^{\heart, \ge n+1}_\mathrm{FG}}(\pi_t(\sF))$  is zero whenever $s > r$, where $r=r(\pi_t(\sF))$ is
is the degree of buds of formal groups from the moduli stack of which the quasi-coherent sheaf $\pi_t(\sF)$ is pulled back along, in the sense of \cite[Sections 3.3 \& 3.4]{Goerss}. That such a number $r(\pi_t(\sF))$ exists is a consequence of \cite[Theorem 3.27]{Goerss} and the coherence hypothesis on $\pi_t(\sF)$.

Hence for fixed $s, t$, the system $(E\langle n\rangle_2^{s,t})_n$ is indeed pro-isomorphic to zero. We claim that the same holds for  $(E\langle n\rangle_r^{s,t})_n$ for any $2\le r \le\infty$. This is of particular interest for $r=\infty$, for which we conclude that
$(E\langle n\rangle_\infty)_n = (\pi_*(\Gamma_{\mathcal M^{\heart, \ge n}_\mathrm{FG}}(\sF)))_n$ is  pro-isomorphic to zero. Since any pro-trivial filtered system clearly satisfies the Mittag-Leffler property, it follow that
$$
0 \simeq \varprojlim_n \pi_i(\Gamma_{\mathcal M^{\heart, \ge n}_\mathrm{FG}}(\sF))\simeq \pi_i\big(\varprojlim_n \Gamma_{\mathcal M^{\heart, \ge n}_\mathrm{FG}}(\pi_t(\sF)) \big)
$$
for all $i\in \mathbf Z$, leading to the desired conclusion that $\varprojlim_n \Gamma_{\mathcal M^{\heart, \ge n}_\mathrm{FG}}(\pi_t(\sF))\simeq 0$.

Hence it suffices, thanks to convergence of the local cohomology spectral sequences, to show that the filtered system $(\mathcal H^s_{\mathcal M^{\heart, \ge n}_\mathrm{FG}}(\pi_t(\sF)))_n$, which we already know to be pro-isomorphic to zero, vanishes uniformly in $s$ and $t$. Recalling the above argument, we see that, for fixed $s$ and $t$, the pro-vanishing occurs in degrees $n >  r(\pi_t(\sF))$ and all degrees $s$. Let us set $r = \mathrm{sup}_{t\in \mathbf Z}\,r(\pi_t(\sF))\le \infty$. If $r<\infty,$ then can we conclude that that the pro-vanishing occurs at $n > r$ for all $s$ and $t$ simultaneously as desired.

It remains only to verify that indeed $r <\infty$, which is to say, that the sequence of non-negative integers $(r(\pi_t(\sF))_{t\in \mathbf Z}$ is bounded.
Since  $\sF$ is a quasi-coherent sheaf on the stack of oriented formal groups, where $\Sigma^{-2}(\sO_{\M})\simeq \omega_{\M}$, we have a canonical isomorphism
$$
\pi_{t+2i}(\sF)\simeq \pi_t(\sF)\o_{\mathcal M^\heart_\mathrm{FG}}\omega_{\mathcal M_\mathrm{FG}^\heart}^{\o i}
$$
of ordinary coherent sheaves on $\mathcal M^\heart_\mathrm{FG}$. Consequently, all the homotopy sheaves $\pi_t(\sF)$ are obtained from $\pi_0(\sF)$ and $\pi_1(\sF)$ by tensoring with the line bundle of invariant differentials $\omega_{\mathcal M^\heart_\mathrm{FG}}$. The latter is a coherent sheaf on $\mathcal M^\heart_\mathrm{FG}$, so by \cite[Theorem 3.27]{Goerss} there exists a number $r_\omega = r(\omega_{\mathcal M^\heart_\mathrm{FG}})$ as above.  It follows from the fact that buds of formal groups form a tower, and quasi-coherent pullback being symmetric monoidal, that $r(\sG\o\sG') = \mathrm{max}(r(\sG), r(\sG'))$ holds for any pair of coherent sheaves $\sG, \sG'\in\QCoh(\mathcal M^\heart_\mathrm{FG}\o_{\mathbf Z}\mathbf Z_{(p)})^\heart$. Consequently, for any $t\in \mathbf Z$ and $i\ge 1$, we have
$$
r(\pi_{t+2i}(\sF))=\mathrm{max}(r(\pi_t(\sF)), r(\omega_{\mathcal M^\heart_\mathrm{FG}})),
$$
from which it is clear that $r(\pi_t(\sF))$ depends only on the parity of $t$. It follows that
$$r= \sup_{t\in \mathbf Z}r(\pi_t(\sF)) = \max(r(\pi_0(\sF)), r(\pi_1(\sF)),  r(\omega_{\mathcal M^\heart_\mathrm{FG}}))$$
is a finite number, concluding the proof.
\end{proof}

\begin{remark}
In the notation of Remark \ref{moduli stack of finite height},  we may express Theorem \ref{CCT} as asserting that quasi-coherent sheaves on $\M\o S_{(p)}$ with coherent cohomology are obtained by quasi-coherent pushforward along the canonical map $\mathcal M^{\mathrm{or},< \infty}_\mathrm{FG}\to \M\o S_{(p)}$.
\end{remark}

\begin{corollary}[Hopkins-Ravenel Chromatic Convergence Theorem]
let $X$ be a finite $(p)$-local spectrum. Then the chromatic filtration on $X$ is complete, in the sense that it exhibits an equivalence of spectra $X\simeq \varinjlim_n L_nX$.
\end{corollary}

\begin{proof}
This follows from Theorem \ref{CCT} and the fact that the quasi-coherent sheaves $\pi_t(\sO_{\M}\o X)$ on $\mathcal M^\heart_\mathrm{FG}\o_{\mathbf Z} \mathbf Z_{(p)}$ coherent for every $t\in\mathbf Z$. Indeed, by definition of homotopy sheaves \cite[Definition 1.4.7]{ChromaticCartoon}, and using the flat cover $\Spec(\MP_{(p)})\to \M\o S_{(p)}$ of (a $(p)$-local version of) \cite[Lemma 2.3.5]{ChromaticCartoon}, the quasi-coherent sheaf $\pi_t(\sO_{\M}\o X)$ is coherent if and only if $\mathrm{MP}_i(X)$ is a coherent $\pi_0(\mathrm{MP}_{(p)}) = L_{(p)}$-module. Consider the full subcategory $\mathcal C\subseteq\Sp_{(p)}$, spanned by all such $(p)$-local spectra for which $\pi_t(\sO_{\M}\o X)$ is a finitely presented $L_{(p)}$-module. The $(p)$-complete Lazard ring $L\simeq \mathbf Z_{(p)}[x_1, x_2, \ldots]$ is a coherent commutative ring by \cite[Example  7.2.4.14]{HA}, hence \cite[Lemma 7.2.4.15]{HA} implies that $\mC$ is a thick subcategory of $\Sp_{(p)}$. It also contains the $(p)$-local sphere $S_{(p)}$, hence it must contain the thick subcategory spanned by it. But that is precisely the $\i$-category of finite $(p)$-local spectra.
\end{proof}

\subsection{The monochromatic layer and $K(n)$-local spectra}
Recall e.g.\ \cite[Lecture 34]{Lurie Chromatic} from that the \textit{$n$-th monochromatic layer} of a spectrum $X$ is defined to be the fiber of the chromatic filtration
$$
M_nX :=\fib(L_nX\to L_{n-1}X).
$$
The functor $M_n$ is an idempotent projector onto the full subcategory $M_n\Sp\subseteq L_n\Sp$ of \textit{monochromatic spectra of height $n$}, i.e.\ of $E(n)$-local spectra for which the canonical map $X\to M_nX$ is an equivalence. This is equivalent to the spectrum $X$ being both $E(n)$-local and $E(n-1)$-acyclic. Thanks to both $L_n$ and $L_{n-1}$ being smashing localizations, it follows that $M_nX\simeq X\o M_n S,$ and in particular, the subcategory $M_n\Sp\subseteq\Sp$ inherits a symmetric monoidal structure from the smash product of spectra.

\begin{prop}\label{monochrom from qcoh}
The adjunction of Theorem \ref{Chromatic localization is restriction}
exhibits an equivalence of  symmetric monoidal $\i$-categories
$
\QCoh_{\mathcal M^{\heart, =n}_\mathrm{FG}}(\mathcal M^{\mathrm{or}, \le n}_\mathrm{FG})\simeq M_n\mathrm{Sp}
$
between quasi-coherent sheaves on $\mathcal M^{\mathrm{or}, \le n}_\mathrm{FG},$ supported along the reduced closed substack $\mathcal M_\mathrm{FG}^{\heart, =n}\subseteq\mathcal M_\mathrm{FG}^{\heart, \le n}$, and monochromatic spectra of height $n$. 
\end{prop}

\begin{proof}
By Theorem \ref{Chromatic localization is restriction}, we already know that said adjunction is a symmetric monoidal adjoint equivalence of $\i$-categories $\QCoh(\Mn)\simeq L_n\Sp$. Both $\QCoh_{\Men}(\Mn)\subseteq\QCoh(\Mn)$ and $L_n\Sp\subseteq\Sp$ are defined analogously: they are the full subcategories of objects $\sF$ and $X$ respectively, for which the canonical map, $\Gamma_{\Mn}(\sF)\to \sF$ and $M_n X\to L_nX\simeq X$ respectively, is an equivalence. The result now follows from the next Lemma, relating the two functors in question.
\end{proof}

\begin{lemma}\label{monochrom from qcoh intro}
For any $(p)$-local spectrum $X$, the canonical map of spectra
$$
M_nX\to \Gamma_{\mathcal M_\mathrm{FG}^{\heart, =n}}(\mathcal M^{\mathrm{or}, \le n}_\mathrm{FG}; \sO_{\M}\o X).
$$
is an equivalence, expressing the $n$-th monochromatic layer of $X$ as supported sections of quasi-coherent sheaves on $\Mn$.
\end{lemma}
\begin{proof}
For any quasi-coherent sheaf $\sF$ on $\M\o S_{(p)}$, Proposition \ref{support and qc pullback} gives for an arbirary quasi-coherent sheaf $\sF$ on $\M\o S_{(p)}$ a canonical equivalence
$$
\Gamma_{\mathcal M^{\heart, =n}_\mathrm{FG}}(\sF\vert_{\mathcal M^{\mathrm{or}, \le n}_\mathrm{FG}})\simeq \Gamma_{\mathcal M^{\heart, \ge n}_\mathrm{FG}}(\sF)\vert_{\mathcal M^{\mathrm{or}, \le n}_\mathrm{FG}}
$$
in the $\i$-category $\QCoh(\mathcal M^{\heart, \le n}_\mathrm{FG})$. At the same time, Corollary \ref{corollary semi-orthogonal sequence} gives rise to a cofiber sequence
$$
\Gamma_{\mathcal M^{\heart, \ge n}_\mathrm{FG}}(\sF)\to \sF\to \sF\vert_{\mathcal M^{\mathrm{or}, \ge n-1}_\mathrm{FG}}
$$
of quasi-coherent sheaves on $\M\o S_{(p)}$. Since quasi-coherent pullback along the open immersion $\mathcal M^{\mathrm{or},\le n}_\mathrm{FG}\to\M\o S_{(p)}$ is an exact functor, it  produces from the previous one another cofiber sequence
$$
\Gamma_{\mathcal M^{\heart, =n}_\mathrm{FG}}(\sF\vert_{\mathcal M^{\mathrm{or}, \le n}_\mathrm{FG}})\to \sF\vert_{\mathcal M^{\mathrm{or}, \le n}_\mathrm{FG}}\to \sF\vert_{\mathcal M_\mathrm{FG}^{\mathrm{or}, \le n-1}}
$$
in the $\i$-category $\QCoh(\mathcal M^{\heart, \le n}_\mathrm{FG})$. Under the equivalence of the latter $\i$-category via the global sections functor $\Gamma(\mathcal M^{\mathrm{or}, \le n}_\mathrm{FG}; -)$ with the $\i$-category $L_n\Sp$, and the explicit formula for chromatic localization from Corollary \ref{L_n as restriction}, we obtain for the constant quasi-coherent sheaf $\sF :=\sO_{\M}\o X$ on $\M\o S_{(p)}$ the fiber sequence of spectra
$$
\Gamma_{\mathcal M^{\heart, =n}_\mathrm{FG}}(\mathcal M^{\mathrm{or}, \le n}_\mathrm{FG}; \sO_{\M}\o X)\to L_n X\to L_{n-1}X,
$$
exhibiting its left-most term as the $n$-th monochromatic layer of $X$ as desired.
\end{proof}

An important result in chromatic homotopy theory, e.g.~\cite[Lecture 34, Proposition 12]{Lurie Chromatic} is that monochromatic spectra of height $n$ may equivalently be understood as $K(n)$-local spectra. The latter is a completeness condition, and we will explain how it fits into the comparison between the quasi-coherent sheaves with prescribed support conditions and quasi-coherent sheaves on formal completions as worked out in Proposition \ref{all about formal QCoh}.

\begin{prop}\label{mono vs K(n)-loc 1}
The adjunction between completion and supported sections on quasi-coherent sheaves
$$
(-)^\wedge_{\mathcal M^{\heart, =n}_\mathrm{FG}} : \QCoh_{\mathcal M^{\heart, =n}_\mathrm{FG}}(\mathcal M^{\mathrm{or}, \le n}_\mathrm{FG})\simeq\QCoh\big ((\mathcal M^{\mathrm{or}, \le n}_\mathrm{FG})^\wedge_{\mathcal M^{\heart, =n}_\mathrm{FG}}\big ) : \Gamma_{\mathcal M^{\heart, =n}_\mathrm{FG}}
$$
is a symmetric monoidal adjoint equivalence of $\i$-categories.
\end{prop}

\begin{proof}
By Proposition \ref{all about formal QCoh}, it suffices to show that for geometric non-connective spectral stack $\mathcal M^{\mathrm{or}, \le n}_\mathrm{FG}$,  
the map $\mathcal M_\mathrm{FG}^{\heart, =n}\to\mathcal M_\mathrm{FG}^{\heart, \le n}$ is a finitely presented closed immersion of a reduced closed substack. The only part of this that we do not yet know is that the finite presentation assertion. To check it, it suffices to prove it for pullback along any map from an ordinary affine $\Spec(R)\to \mathcal M^{\heart, \le n}_\mathrm{FG}$. Such a map corresponds to a formal group $\w{\G}$ of height $\le n$ over the $(p)$-local ordinary commutative ring $R$, and the pullback $\Spec(R)\times_{\mathcal M^{\heart, \le n}_\mathrm{FG}}\mathcal M^{\heart, = n}_\mathrm{FG}$ is the closed locus of points where it  has height exactly $n$. By the definition of the $n$-th Landweber ideal from Remark \ref{Landideal}, this is precisely the vanishing locus of $\mathfrak I^{\w{\G}}_n\subseteq R$. But this is indeed a finitely generated ideal, since we can (after a number of choices) write $\mathfrak I^{\w{\G}}_n=(p, v_1, \ldots, v_n)$, see \cite[Proposition 4.4.10]{Elliptic 2}.
\end{proof}

In order to (directly, i.e.\ not using the above result) related the right-hand side of the equivalence of Proposition \ref{mono vs K(n)-loc 1} with the $K(n)$-local stable category, we will need to gain a better understanding of the algebro-geometric objects involved. This will be accomplished throughout the next couple lemmas.

In what follows, let  $\Spec(\overline{\mathbf F}_p)\to \Men$ be a (unique up to isomorphism) map of formal stacks, classifying a formal group $\w{\G}$ of exact height $n$ over the algebraic closure $\overline{\mathbf F}_p$ of the prime field $\mathbf F_p$. 
The \textit{(big or extended) Morava stabilizer group} $\mathbb G_n$ is the profinite group viewed as a group scheme; see \cite[Remark 2.12]{DevHop}. As also discussed there, we view it either as existing over $\Spec(\mathbf Z)$ in the realm of ordinary algebraic geometry, or over $\Spec(S)$ in the realm of spectral algebraic geometry, to which we extend it by decompression, i.e.\ by pre-composing with the functor $\pi_0:\CAlg\to\CAlg^\heart$.

\begin{lemma}\label{lemma ichi}
 The induced map of formal completions $(\Mn)^\wedge_{\Spec(\F)}\to (\Mn)^\wedge_{\Men}$ in the $\i$-category $\Fun(\CAlg_\mathrm{ad}^\mathrm{cpl}, \mS)$ exhibits an equivalence
$$(\Mn)^\wedge_{\Men}\simeq (\Mn)^\wedge_{\Spec(\F)}/\mathbb G_n$$
with the quotient under an action of the Morava stabilizer group $\mathbb G_n$.
\end{lemma}

\begin{proof}
We first prove the analogous assertion on the level of ordinary stacks.
By Definition \ref{completion def}, the adic incompletion of the moduli $\Mon$ along $\Men$ is given by
$$
(\Mon)_{\Men}\simeq \Mon \times_{(\Mon)_\mathrm{dR}}(\Men)_\mathrm{dR}.
$$
The map of ordinary stacks $\Spec(\overline{\mathbf F}_p)\to \Men$ gives rise by the classical Theorem \ref{orbit picture at height n} to an equivalence with the quotient stack $\Men\simeq \Spec(\overline{\mathbf F}_p)/\underline{\mathrm{Aut}}(\w{\G})$, under the automorphism group scheme $\underline{\mathrm{Aut}}(\w{\G})$ of the exact height $n$ formal group $\w{\G}$ over $\F$. The quotient here is taken in the $\i$-category of ordinary stacks over $\Spec(\overline{\mathbf F}_p)$. Since the de Rham functor preserves both colimits and finite limits, this implies an equivalence of de Rham spaces $(\Men)_\mathrm{dR}\simeq\Spec(\F)_\mathrm{dR}/\underline{\mathrm{Aut}}(\w{\G})_\mathrm{dR}$, with the quotient on the right formed over $\Spec(\F).$ Consequently
\begin{eqnarray*}
(\Mon)_{\Men}
&\simeq&
\Mon \times_{(\Mon)_\mathrm{dR}}(\Men)_\mathrm{dR}\\
&\simeq&
\Mon\times_{(\Mon)_\mathrm{dR}}\Spec(\F)_{\mathrm{dR}}/\underline{\mathrm{Aut}}(\w{\G})_\mathrm{dR}\\
&\simeq&
(\Mon)_{\Spec(\F)}/\underline{\mathrm{Aut}}(\w{\G})_\mathrm{dR},
\end{eqnarray*}
where the adic pre-completion in the final term is taken along the composite map $\Spec(\F)\to\Men\to\Mon$.
In \cite[Lemma 2.11]{DevHop}, we showed (see Remark \ref{honesty} that, when restricting all the functors to complete adic $\E$-rings $\CAlg_\mathrm{ad}^\mathrm{cpl}\subseteq\CAlg_\mathrm{ad}$, there is a canonical equivalence
\begin{equation}\label{a lemma borrowed}
\mathrm{Aut}(\w{\G})_\mathrm{dR}\simeq \mathbb G_n\times\Spec(\F)_\mathrm{dR}.
\end{equation}
 Since we defined formal completion of non-connective adic stacks in Definition \ref{real def of formal completion} precisely as restriction to the subcategory $\CAlg_\mathrm{ad}^\mathrm{cpl}\subseteq\CAlg_\mathrm{ad}$, we obtain by combining all of the above an equivalence of formal completions
$$
(\Men)^\wedge_{\M}\simeq (\Mon)_{\Spec(\F)}^\wedge/\mathbb G_n,
$$
with the quotient on the right formed `absolutely' i.e.\ over $\Spec(\mathbf Z)$.
Since we get from Remark \ref{fc is decompress} that
$$
(\Mn)^\wedge_{\Men}\simeq \Mn\times_{\Men\circ\pi_0} (\Men)^\wedge_{\M},
$$
and the same for formal completion along $\Spec(\F)$, the desired equivalence follows.
\end{proof}

\begin{remark}\label{honesty}
Technically, we proved the equivalence \eqref{a lemma borrowed} in \cite[Lemma 2.11]{DevHop} only when restricted to a smaller subcategory of complete adic $\E$-rings $\CAlg^\mathrm{cN}_{/\F}\subseteq\CAlg^\mathrm{cpl}_\mathrm{ad}$, see \cite[Definition 2.1]{DevHop}. But since the Lubin-Tate spectrum $E_n = E(\F, \w{\G})$ belongs to this subcategory, the simplicial presentation of $\Mn$ from Remark \ref{Cech nerve for len with LT} shows that all the functors from $\CAlg_\mathrm{ad}^\mathrm{cpl}$ in the preceding discussion are left Kan extended from this subcategory, legitimizing the above proof.
\end{remark}

\begin{lemma}\label{lemma ni}
Let $\w{\G}_0$ be a formal group of (exact) height $n$ over a perfect field $\kappa$ of characteristic $p$. 
The map of non-connective spectral stacks $\Spec(E(\kappa, \w{\G}_0))\to\Mn$ exhibits an equivalence
$$
(\Mn)^\wedge_{\Spec(\kappa)}\simeq\Spf (E(\kappa, \w{\G}_0)).
$$
in the $\i$-category $\Fun(\CAlg_\mathrm{ad}^\mathrm{cpl}, \mS)$.
\end{lemma}

\begin{proof}
The Lubin-Tate $\E$-ring $E(\kappa, \w{\G}_0)$ is obtained by \cite[Theorem  5.1.5,  Remark  6.0.7]{Elliptic 2} as follows: it is the $K(n)$-localization of the oriented universal deformation $\E$-ring $R_{\w{\G}_0}^\mathrm{or}$, which in turn satisfies a universal property. That is, $\Spf(R_{\w{\G}_0}^\mathrm{or})$ classified oriented deformations of the formal group $\w{\G}_0$, i.e.\ for any complete adic $\E$-ring we have
$$
\Map_{\CAlg_\mathrm{ad}^\mathrm{cpl}}(R_{\w{\G}_0}^\mathrm{or},A)\simeq \varinjlim_{I\subseteq\pi_0(A)}\mathcal M^\mathrm{or}_\mathrm{FG}(A)\times_{\mathcal M^\heart_\mathrm{FG}(\pi_0(A)/I)} \{\w{\G}_0\},
$$
with the colimit ranging over all finitely generated ideals of definition $I\subseteq\pi_0(A)$. In light of our definition of formal completion, see Remark \ref{fc explicit}, we may identify this with the formal completion of the non-connective spectral stack of formal group $\mathcal M^\mathrm{or}_\mathrm{FG}$ along the map of ordinary stacks
$$
\Spec(\kappa)\xrightarrow{\eta_{\w{\G}_0}}\mathcal M^{\heart, =n}_\mathrm{FG}\to\mathcal M^\heart_\mathrm{FG}\simeq (\mathcal M^\mathrm{or}_\mathrm{FG})^\heart.
$$
Thus $\Spf(R^\mathrm{or}_{\w{\G}_0})\simeq (\M)^\wedge_{\Spec(\kappa)}$, and it remains to understand the effect of the $K(n)$-completion.
 Both the functors in the canonical map $\Spf(E(\kappa, \w{\G}_0))\to\Spf(R^\mathrm{or}_{\w{\G}_0})$ return the empty set for  all $\E$-ring which are not complex oriented, but the first one also does so unless the oriented $\E$-ring is $K(n)$-local.
  In fact, since Lubin-Tate spectra are always $(p)$-complete (as follows from the inspection of their coefficients), we may further factor the construction through $(p)$-completion, and rather  consider instead the map $\Spf(E(\kappa, \w{\G}_0))\to\Spf\big((R^\mathrm{or}_{\w{\G}_0})_{(p)}\big)$. 
  Since we are already assuming the adic $\E$-rings to be complete, it follows from \cite[Theorem 4.5.2]{Elliptic 2} that $K(n)$-locality of the $(p)$-local adic $\E$-rings will amount to their Quillen formal groups being of height $\le n$. That is to say, we find that
\begin{eqnarray*}
\Spf(E(\kappa, \w{\G}_0))
&\simeq&
\Mn \times_{\M\o S_{(p)}}\Spf\big( (R^\mathrm{or}_{\w{\G}_0})_{(p)}\big)\\
&\simeq&
\Mn \times_{\M\o S_{(p)}} (\M\o S_{(p)})^\wedge_{\Spec(\kappa)}\\
&\simeq &
\Mn \times_{\M\o S_{(p)}} \M\o S_{(p)} \times_{\big((\mathcal M^\heart_\mathrm{FG}\o_{\mathbf Z}\mathbf Z_{(p)})_\mathrm{dR}\circ \pi_0\big)} (\Spec(\kappa)_\mathrm{dR}\circ\pi_0)\\
&\simeq&
\Mn \times_{\big((\mathcal M^\heart_\mathrm{FG}\o_{\mathbf Z}\mathbf Z_{(p)})_\mathrm{dR}\circ \pi_0\big)} (\Spec(\kappa)_\mathrm{dR}\circ\pi_0).
\end{eqnarray*}
Both of the maps in the final fiber product factor through the de Rham-ification of $\Mon\to \Mo\o_{\mathbf Z}\mathbf Z_{(p)},$ composed with $\pi_0$. Since that is an open immersion, and so $\Mon\times_{\Mo\o_{\mathbf Z}\mathbf Z_{(p)}}\Mon\simeq \Mon$, the fiber product may be taken with either base. Hence
$$
\Spf(E(\kappa, \w{\G}_0))\simeq \Mn\times_{((\Mon)_\mathrm{dR}\circ \pi_0)}(\Spec(\kappa)_\mathrm{dR}\circ \pi_0)\simeq (\Mn)^\wedge_{\Spec(\kappa)},
$$
arriving at the desired equivalence.
\end{proof}

We are at last in a position to compare the $\i$-category of quasi-coherent sheaves, appearing on the right-hand side of the equivalence in Proposition \ref{mono vs K(n)-loc 1}, with the $K(n)$-local stable category.

\begin{theorem}\label{knloc}
The adjunction
$$
\sO_{(\Mn)^\wedge_{\Men}}\o - : \Sp \rightleftarrows \QCoh\big((\Mn)^\wedge_{\Men}\big) : \Gamma\big((\Mn)^\wedge_{\Men};-\big)
$$
induces a symmetric monoidal equivalence of $\i$-categories
$$
 \QCoh\big((\Mn)^\wedge_{\Men}\big)\simeq L_{K(n)}\Sp
$$
between the quasi-coherent sheaves on the formal completion of $\Mn$ along the closed ordinary substack $\Men\subseteq\Mon$, and $K(n)$-local spectra.
\end{theorem}

\begin{proof}
Combining Lemmas \ref{lemma ichi} and \ref{lemma ni}, we obtain an equivalence
$$
(\Mn)^\wedge_{\Men}\simeq \Spf(E_n)/\mathbb G_n.
$$
in the $\i$-category $\Fun(\CAlg^\mathrm{cpl}_\mathrm{ad}, \mS)$. Passage to the $\i$-categories of quasi-coherent sheaves induces an equivalence of $\i$-categories
$$
\QCoh\big((\Mn)^\wedge_{\Men}\big)\simeq \QCoh(\Spf(E_n)/\mathbb G_n)\simeq \Mod_{E_n}^{h\mathbb G_n}.
$$
The statement of the theorem now reduces to \cite[Corollary 2.17]{DevHop} or \cite[Proposition 10.10]{Mathew}.
\end{proof}

The following consequences  of the Theorem are straightforward, following analogously to how Corollaries  \ref{L_n as restriction} and \ref{C_n as local cohomology} follows from Theorem \ref{Chromatic localization is restriction}.

\begin{corollary}
In terms of the equivalence of $\i$-categories of Theorem \ref{Chromatic localization is restriction},  the $K(n)$-localization of a $(p)$-local spectrum $X$ corresponds to the completion in quasi-coherent sheaves $(\sO_{\Mn}\o X)^\wedge_{\Men}$.
\end{corollary}

\begin{corollary}
The adjoint equivalence between quasi-coherent sheaves of Proposition \ref{mono vs K(n)-loc 1} induces upon global sections, and is  by Proposition \ref{monochrom from qcoh}  and Theorem \ref{knloc} equivalent to, the adjoint equivalence of stable $\i$-categories
$$L_{K(n)}:M_n\Sp\simeq L_{K(n)}\Sp : M_n.$$
\end{corollary}

\end{document}